\documentclass[a4paper,10pt,twoside]{article}
\usepackage{wiaspreprint}
\usepackage{amsmath,amssymb,amsthm,mathptmx,hyperref,nicefrac}
\usepackage{titlesec}
\numberwithin{equation}{section}
\newtheorem{theorem}{Theorem}[section]
\theoremstyle{remark}
\newtheorem{remark}[theorem]{Remark}
\theoremstyle{definition}
\newtheorem{definition}[theorem]{Definition}
\newtheorem{proposition}[theorem]{Proposition}
\newtheorem{lemma}{Lemma}[section]
\newtheorem{corollary}[theorem]{Corollary}
\newcommand{\f}[1]{\pmb{#1}}

\DeclareMathOperator{\N}{\mathbb{N}}
\DeclareMathOperator{\R}{\mathbb{R}}

\DeclareMathOperator{\C}{\mathcal{C}}
\DeclareMathOperator{\F}{\mathcal{F}}
\DeclareMathOperator{\GY}{\text{GY}}

\DeclareMathOperator{\V}{\f H^1_{0,\sigma}}
\DeclareMathOperator{\Vd}{(\f H^{1}_{0,\sigma})^*}

\DeclareMathOperator{\Ha}{\f L^2_{\sigma}}
\DeclareMathOperator{\Hb}{\f{H}^1_0}
\DeclareMathOperator{\Hc}{\f H^2}
\DeclareMathOperator{\Hd}{\f H^{-1}}
\DeclareMathOperator{\He}{\f{H}^1}
\DeclareMathOperator{\Le}{\f{L}^2}

\DeclareMathOperator{\sgn}{sgn}

\newcommand{\Hrand}[1]{{\f H^{\nicefrac{#1}{2}}(\partial\Omega)}}

\DeclareMathOperator{\Sr}{\mathbb{E}}

\DeclareMathOperator{\S2}{\mathbb{S}^2}
\DeclareMathOperator{\Se}{\mathcal{S}}

\DeclareMathOperator{\M}{\mathcal{M}}

\DeclareMathOperator{\ra}{\rightarrow}
\DeclareMathOperator{\de}{\text{d}}
\DeclareMathOperator{\tr}{tr}

\newcommand{\dreidots}{\text{\,\multiput(0,-2)(0,2){3}{$\cdot$}}\,\,\,\,}

\newcommand{\dreidotkom}{\text{\,\multiput(0,0)(0,2){2}{$\cdot$}\put(0,0){,}}\,\,\,\,}

\newcommand{\pat}[2]{\frac{\partial #1}{\partial #2}}
\DeclareMathOperator{\di}{\nabla \cdot}

\newcommand{\ov}[1]{\overline{#1}}

\newcommand{\rot}[1]{[ #1 ]_{\f X}}
\newcommand{\rott}[1]{[ #1 ]_{-\f X}}
\DeclareMathOperator{\curl}{\nabla \times}

\newcommand{\intte}[1]{\int_{0}^T{ #1} \de t}
\newcommand{\inttet}[1]{\int_{0}^{t}{ #1} \de s}
\newcommand{\inttett}[1]{\int_{0}^{t}\left [{ #1} \right ]\de s}

\renewcommand{\ll}[1]{\langle\hspace{-0.75mm}\langle{#1}\rangle\hspace{-0.75mm}\rangle}

\DeclareMathOperator{\sym}{{sym}}
\DeclareMathOperator{\Lap}{\Delta_{\f \Lambda}}

\DeclareMathOperator{\skw}{skw}
\newcommand{\sy}[1]{(\nabla \f #1)_{{\sym}}}
\newcommand{\sk}[1]{(\nabla \f #1)_{\skw}}
\renewcommand{\t}{\partial_t  }

\newcommand{\fn}[1]{\f {{#1}}_{n}}

\newcommand{\vv}{\tilde{\f v}}
\newcommand{\dd}{\tilde{\f d}}

\newcommand{\syv}{(\nabla \tilde{\f v})_{{\sym}}}

\newcommand{\skv}{(\nabla \tilde{\f v})_{\skw}}

\renewcommand{\o}{\otimes}

\newcommand{\tq}{\tilde{\f q}}
\newcommand{\te}{\tilde{\f e}}
\newcommand{\HH}{\tilde{\f H}}

\begin{document}
\author[R. Lasarzik]{
Robert Lasarzik\nofnmark\footnote{Weierstrass Institute \\
Mohrenstr. 39 \\ 10117 Berlin \\ Germany \\
E-Mail: robert.lasarzik@wias-berlin.de}}

\title[Dissipative solution to the Ericksen--Leslie system equipped with the Oseen--Frank energy]{Dissipative solution to the Ericksen--Leslie system\\equipped with the Oseen--Frank energy
}	
\nopreprint{2508}	
\selectlanguage{english}		
\maketitle
\begin{abstract}
We analyze the Ericksen--Leslie system equipped with the Oseen--Frank energy in three space dimensions. 
The new concept of dissipative solutions is introduced. 
Recently, the author introduced the concept of measure-valued solutions to the considered system and showed global existence as well as weak-strong uniqueness of these generalized solutions. In this paper, we show that the expectation of the measure valued solution is a dissipative solution. 
The concept of  a dissipative solution itself relies on an inequality instead of an equality, but is described by functions instead of parametrized measures. 
 These solutions exist globally and fulfill the weak-strong uniqueness property. 
 Additionally, we generalize the relative energy inequality to solutions fulfilling different nonhomogeneous Dirichlet boundary conditions and incorporate the influence of a temporarily constant electromagnetic field.  
Relying on this generalized energy inequality,  we investigate the long-time behavior and show that  all solutions converge for the large time limit to a certain steady state. 
\end{abstract}

\setcounter{tocdepth}{2}
\tableofcontents


\section{Introduction}\label{sec:intro}
Nonlinear partial differential equations require generalized solution concepts.  
In a recent series of articles, we introduced the concept of measure-valued solutions to the Ericksen--Leslie system equipped with the Oseen--Frank energy, showed their global existence~\cite{masswertig} and weak-strong uniqueness~\cite{weakstrong}. To show the weak-strong uniqueness, the relative energy approach was generalized to nonconvex functionals (see~\cite{weakstrong} and~\cite{sabine}). 
In the article at hand, we take another path towards generalized solution concepts. Instead of generalizing the solution concept by going from functions to parametrized measures (so-called generalized Young measures), we require that the solution fulfills an inequality instead of an equality.
Therewith, one ends up with so-called \textit{dissipative solutions}.

The concept of dissipative solutions was first introduced by Pierre Louis Lions in the context of the Euler equations~\cite[Sec.~4.4]{lionsfluid} with ideas originating from the Boltzmann equation~\cite{LionsBoltzman}. 
It is also applied in the context of incompressible viscous
electro-magneto-hydro\-dy\-na\-mics (see Ars\'{e}nio and Saint-Raymond~\cite{raymond}), the two-component Camassa--Holm system~\cite{dissCamassa}, and equations of viscoelastic diffusion in polymers~\cite{viscoelsticdiff}.
The term dissipative solutions is also used for solutions to the Navier--Stokes--Fourier system (see Feireisl~\cite{feireislsingular}), whereas this concept is a little different from the dissipative solutions in this paper or in~\cite{lionsfluid} (see Remark~\ref{rem:feireisl}). 

All these solution concepts rely on the appropriate estimate of an associated relative entropy. 
In our context, the main quantities of interest are (free) energies rather than entropies. We therefore refer henceforth to a "relative energy".
In the context of thermodynamics, the idea of a relative energy to compare two solutions goes back to Dafermos~\cite{dafermos}.
 For a convex G\^{a}teaux differentiable  energy function $\eta: V \ra \R$ for a Banach space $V$, the relative energy of two solutions~$u$ and $\tilde{u}$ is given by (see Ref.~\cite[Sec.~5.3]{dafermos2})
\begin{align}
\mathcal{E}:= \eta (u) - \eta(\tilde{u}) - \langle\eta'(\tilde{u}), u -\tilde{u}\rangle \,.\label{relencon}
\end{align}
The  convexity of $\eta$ guarantees that $\mathcal{E}$ is nonnegative (see~\cite[Kapitel~III,Lemma~4.10]{groeg}). 
One novelty of the case presented in this paper is that the considered energy functional is nonconvex.
  For nonconvex functionals, the quantity given in~\eqref{relencon} is not necessarily positive anymore. 
In~\cite{weakstrong}, we provided a remedy by introducing a relative energy for the nonconvex case of the Oseen--Frank energy. 
In the article at hand, this relative energy is adapted in order to introduce the concept of dissipative solutions to the Ericksen--Leslie system equipped with the Oseen--Frank energy describing nematic liquid crystal flow.

Nematic liquid crystals are anisotropic fluids. The rod- or disk-like molecules build, or are dispersed in, a fluid and are directionally ordered. 
This ordering and its direction heavily influences the properties of the material such as light scattering or rheology. This gives rise to many applications, where \textit{liquid crystal displays} are only the most prominent ones. 


Due to its simplicity and its good agreement with experiments, the \textit{Ericksen--Leslie model} is the most common model to describe nematic liquid crystals. 
In nematic liquid crystals, the molecules tend to be aligned in a common direction, at least in equilibrium situations. This predominant direction is described by a unit vector, the so-called director, henceforth denoted by $\f d$. 
The director can be seen as the local average of the directions over a set of molecules. 
Since the rod-like molecules show a head-to-tail symmetry, the directions $\f d$ and $- \f d$ are not distinguishable. 

In the mathematics community, there has been quit some work on the existence of strong and weak solutions to several simplifications of the Ericksen--Leslie model. 
The first mathematical analysis of  a simplified Ericksen--Leslie model is done by Lin and Liu~\cite{linliu1}. They show global existence of weak solutions and local existence of strong solutions. Additionally, they manage to generalise these results to a more realistic model~\cite{linliu3}. They also show partial regularity of weak solutions to the considered system~\cite{linliu2}. Following this work, there have been many articles considering slightly more complicated models, see~\cite{prohl},~\cite{allgemein},~or~\cite{isothermal} for example. 
To the best of the author's knowledge, the only generalization with respect to the free energy potential is performed by Emmrich and the author in~\cite{unsere}. 
There are also results on \textit{local existence of solutions} to realistic models (see for instance~\cite{localin3d},~\cite{recent} or~\cite{Pruess2}). Especially, local strong solutions are known to exists to different simplifications of the system considered in this article.
The full (thermodynamically consistent) Ericksen--Leslie system with the one constant approximation of the Oseen--Frank energy is considered in~\cite{Pruess2}. Whereas, the simplified Ericksen--Leslie system with the full Oseen--Frank energy is studied in~\cite{localin3d}. 
For more on liquid crystals, we refer to Emmrich, Klapp and Lasarzik~\cite{sabine}.

In the physics community, it is well-known that the  nematic order can be destroyed during the evolution of the  material. The uniaxial nematic liquid crystal can become biaxial, i.e., exhibiting two predominant directions~\cite{olmsted}.
Therefore, it is not surprising that a global solution concept relies on generalized Young measures, where a possible biaxiality can be described by a strongly oscillating Gradient resulting in a Young measure. 
From a na\"{i}ve point of view, it is possible that in the evolution of the liquid crystal, a jump from the initialization $\f d$ happens to $-\f d$ since both cannot be distinguished (see~\cite{singul2}). 
The generalized Young measures describes possible concentration effects in the gradient implying  jumps in the function itself. 

From the modeling point of view, one is not interested in these jumps from $\f d $ to $-\f d$. Also, the director was initially modeling the local average of the directions of the molecules. So it seems that the interesting part of the measure-valued solution is the expectation of the Young measure, since it provides the local averaged direction. 
In the paper at hand, we argue that the expectation of the Young measure is indeed a dissipative solution. 
As we explained above, this is in accordance with the modeling. 


Additionally, we show in a future article that dissipative solutions are more suitable for numerical approximations than measure-valued solutions (see~\cite{approx}).

By the introduction of dissipative solutions~(see Definition~\ref{def:diss}), we show that the relative energy inequality is an effective comparing tool for different solutions and as such provides a new meaningful   solvability concept (existence and weak-strong uniqueness are fulfilled, compare to Lions~\cite[Sec.~4.4]{lionsfluid}). 
In the case of convex energy functionals, the relative energy was also used to consider the long-time behavior of solutions (see~\cite{feireislstab}). In the paper at hand, we generalize this to the nonconvex case and discuss implications of the relative energy inequality on the long-time behavior. 
Therefore, we present a generalized relative energy inequality, which  holds for solutions (or test functions) with different boundary values and incorporates the influence of an electromagnetic field. 
The electromagnetic fields as well as the boundary values are assumed to be constant in time.  

The adapted relative energy allows us to show, that all solutions (measure-valued or dissipative) converge for large times to a steady state consisting of  a vanishing velocity field and a director field that fulfills the Euler--Lagrange equations of the Oseen--Frank energy in a measure-valued sense (see Theorem~\ref{thm:longtime}).
%
%


The paper is organised as follows: In Section~\ref{sec:not}, we collect some notation. Section~\ref{sec:model} contains the model and Section~\ref{sec:rel} the definition of dissipative as well as measure-valued solutions and the main result.
The proof of the main result is carried out in Section~\ref{sec:pre}. 
In Section~\ref{sec:5}, we derive the generalized relative energy inequality for functions with different boundary values and incorporating the influence of an electromagnetic field. Finally in Section~\ref{sec:6}, we comment on necessary optimality conditions for the Oseen--Frank energy and investigate the long-time behavior of solutions.


\subsection{Notation\label{sec:not}}
Vectors of $\R^3$ are denoted by bold small Latin letters. Matrices of $\R^{3\times 3}$ are denoted by bold capital Latin letters. We also use tensors of higher order, which are denoted by bold capital Greek letters.
Moreover, numbers are denoted be small Latin or Greek letters, and capital Latin letters are reserved for potentials.
The euclidean scalar product in $\R^3$ is denoted by a dot $ \f a \cdot \f b : = \f a ^T \f b = \sum_{i=1}^3 \f a_i \f b_i$,  for $ \f a, \f b \in \R^3$ and the Frobenius product in $\R^{3\times 3}$ by a double point $ \,\f A: \f B:= \tr ( \f A^T \f B)= \sum_{i,j=1}^3 \f A_{ij} \f B_{ij}$, for $\f A , \f B \in \R^{3\times 3}$.
Additionally, the scalar product in the space of Tensors of order three is denoted by three dots 
\begin{align*}
\f \Upsilon \dreidots\, \f \Gamma : =\left [ \sum_{j,k,l=1} ^3 \f \Upsilon_{jkl} \f \Gamma_{jkl}\right ], \quad    \f \Upsilon \in \R^{3\times 3 \times 3 },  \, \f \Gamma \in \R^{3\times 3 \times 3}  .
\end{align*}
The associated norms are all denoted by $| \cdot |$, where also the norms of tensors of higher order are denoted in the same way 
\begin{align*}
| \f \Lambda|^2 := \sum_{i,j,k,l=1}^3 
\f \Lambda_{ijkl}^2\,,\quad\text{for }\f \Lambda \in \R^{3^4} \quad 
\text{and }\quad| \f \Theta |^2  := \sum_{i,j,k,l,m,n=1}^3 \f \Theta ^2_{ijklmn}\,,\quad\text{for } 
\f \Theta \in \R^{3^6}\,
\end{align*}
respectively.
Similar, we define the products of tensors of different order.
The product of a tensor of third order with a matrix  is defined by
\begin{align*}
\f \Gamma : \f A := \left [ \sum_{j,k=1}^3 \f \Gamma_{ijk}\f A_{jk}\right ]_{i=1}^3\, ,  \,   \f \Gamma \cdot \f A := \left [ \sum_{k=1}^3 \f \Gamma_{ijk}\f A_{kl}\right ]_{i,j,l=1}^3\, ,  \,    \f \Gamma\in \R^{3 \times 3\times 3 } , \, \f A \in \R^{3\times 3}\,.
\end{align*}
The product of a tensor of fourth order with a matrix is defined by
\begin{align*}
\f \Lambda : \f A : =\left [ \sum_{k,l=1} ^3 \f \Lambda_{ijkl} \f A_{kl}\right ]_{i,j=1}^3\,, \, 
 \,    \f \Lambda \in \R^{3^4 },  \,  \f A \in \R^{3\times 3 } \, 
.
\end{align*}
The product of a tensor of fourth order and a matrix or a tensor of third order is defined via
\begin{align*}
 \f A : \f \Theta : ={}& \left [ \sum_{i,j=1} ^3 \f A_{ij} \f \Theta_{ijklmn}  \right ]_{k,l,m,n=1}^3 , \, \f \Theta \dreidots \f \Gamma : = \left [ \sum_{l,m,n=1} ^3 \f \Theta_{ijklmn} \f \Gamma_{lmn}\right ]_{i,j,k=1}^3 ,  \,  
  \f \Theta \in \R^{3^6},\f A \in \R^{3\times 3} ,\f \Gamma \in \R^{3\times 3 \times 3}\,.
\end{align*}
The product of a vector and a tensor of fourth order is defined differently. The definition is adjusted to the cases of this work:
 \begin{align*}
\f a \cdot \f \Theta :={}& \left [ \sum_{k=1} ^3 \f a_{k} \f \Theta_{ijklmn}  \right ]_{i,j,l,m,n=1}^3,
\, 
  \f \Theta \in \R^{3^6},\f a \in \R^{3} \,.
\end{align*}
The standard matrix and matrix-vector multiplication is written without an extra sign for bre\-vi\-ty,
$$\f A \f B =\left [ \sum _{j=1}^3 \f A_{ij}\f B_{jk} \right ]_{i,k=1}^3 \,, \quad  \f A \f a = \left [ \sum _{j=1}^3 \f A_{ij}\f a_j \right ]_{i=1}^3\, , \quad  \f A \in \R^{3\times 3},\,\f B \in \R^{3\times3} ,\, \f a \in \R^3 .$$
The outer vector product is given by
 $\f a \otimes \f b := \f a \f b^T = \left [ \f a_i  \f b_j\right ]_{i,j=1}^3$ for two vectors $\f a , \f b \in \R^3$ and by $ \f A \o \f a := \f A \f a ^T = \left [ \f A_{ij}  \f a_k\right ]_{i,j,k=1}^3 $ for a matrix $ \f A \in \R^{3\times 3} $ and a vector $ \f a \in \R^3$. 
The symmetric and skew-symmetric parts of a matrix are given by 
$\f A_{\sym}: = \frac{1}{2} (\f A + \f A^T)$ and 
$\f A _{\skw} : = \frac{1}{2}( \f A - \f A^T)$, respectively ($\f A \in \R^{3\times  3}$).

We use  the Nabla symbol $\nabla $  for real-valued functions $f : \R^3 \to \R$, vector-valued functions $ \f f : \R^3 \to \R^3$ as well as matrix-valued functions $\f A : \R^3 \to \R^{3\times 3}$ denoting
\begin{align*}
\nabla f := \left [ \pat{f}{\f x_i} \right ] _{i=1}^3\, ,\quad
\nabla \f f  := \left [ \pat{\f f _i}{ \f x_j} \right ] _{i,j=1}^3 \, ,\quad
\nabla \f A  := \left [ \pat{\f A _{ij}}{ \f x_k} \right ] _{i,j,k=1}^3\, .
\end{align*}
 The divergence of a vector-valued and a matrix-valued function is defined by
\begin{align*}
\di \f f := \sum_{i=1}^3 \pat{\f f _i}{\f x_i} = \tr ( \nabla \f f)\, , \quad  \di \f A := \left [\sum_{j=1}^3 \pat{\f A_{ij}}{\f x_j}\right] _{i=1}^3\, .
\end{align*}

For a given tensor of fourth order, we abbreviate the associated second order operator by 
$\Delta_{\f \Lambda}\f d : =\di (\f \Lambda : \nabla \f d)  $ acting on functions $\f d \in \C^2(\Omega \times [0,T];\R^3)$.

Throughout this paper, let $\Omega \subset \R^3$ be a bounded domain of class $\C^{3,1}$.
We rely on the usual notation for spaces of continuous functions, Lebesgue and Sobolev spaces. Spaces of vector-valued or matrix-valued functions are  emphasized by bold letters, for example
$
\f L^p(\Omega) := L^p(\Omega; \R^3)$,
$\f W^{k,p}(\Omega) := W^{k,p}(\Omega; \R^3)$.
The standard inner product in $L^2 ( \Omega; \R^3)$ is just denoted by
$ (\cdot \, , \cdot )$, in $L^2 ( \Omega ; \R^{3\times 3 })$
by $(\cdot ; \cdot )$, and in $L^2 ( \Omega ; \R^{3\times 3\times 3 })$ by   $(\cdot \dreidotkom \cdot )$.

The space of smooth solenoidal functions with compact support is denoted by $\mathcal{C}_{c,\sigma}^\infty(\Omega;\R^3)$. By $\f L^p_{\sigma}( \Omega) $, $\V(\Omega)$,  and $ \f W^{1,p}_{0,\sigma}( \Omega)$, we denote the closure of $\mathcal{C}_{c,\sigma}^\infty(\Omega;\R^3)$ with respect to the norm of $\f L^p(\Omega) $, $ \f H^1( \Omega) $, and $ \f W^{1,p}(\Omega)$, respectively.

The dual space of a Banach space $V$ is always denoted by $ V^*$ and equipped with the standard norm; the duality pairing is denoted by $\langle\cdot, \cdot \rangle$. The duality pairing between $\f L^p(\Omega)$ and $\f L^q(\Omega)$ (with $1/p+1/q=1$), however, is denoted by $(\cdot , \cdot )$, $( \cdot ; \cdot )$, or $( \cdot \dreidotkom \cdot )$. 

The unit ball in d dimensions is denoted by $B_d:= \{ \f x \in \R^d ; | \f x | < 1\}$ and the sphere in $d$-dimensions by  $\Se^{d-1}:= \{ \f x \in \R^d ; | \f d |=1  \}$.

For $Q\subset \R^d$, the Radon measures are denoted by $\mathcal{M}(Q)$, the positive Radon measures by $\mathcal{M}^+(Q)$, and probability measures by $\mathcal{P}(Q)$. We recall that the Radon measures equipped with the total variation are a Banach space and  for compact sets $Q$ , it can be characterized by~$\mathcal{M}(Q) = ( \C(Q))^*$ (see~\cite[Theorem~4.10.1]{edwards}). 
The integration of a function $f\in \C(Q)$ with respect to a measure $\mu\in \mathcal{M}(Q)$ is denoted by $ \int_Qf(\f h ) \mu(\de \f h)\,.$ In case of the Lebesgue measure we just write 
$ \int_Qf(\f h ) \de \f h\,.$

The cross product of two vectors is denoted by $\times $. We introduce the notation $ \rot{\cdot}$, which is defined via
\begin{align}
\rot{\cdot } : \R^d \ra \R^{d\times d}\, , \quad \rot{ \f h} := \begin{pmatrix}
0& - \f h_3 &\f h_2\\
\f h_3 & 0 & - \f h_1 \\
- \f h_2 & \f h_1 & 0
\end{pmatrix}\, .
\end{align}
The $i$-th component of the vector $\f h\in \R^3$ is denoted by $\f h_i$. 
The mapping $\rot{\cdot}$ has some nice properties, for instance
\begin{align*}
\rot{\f a}\f b = \f a \times \f b \, ,\quad \rot{\f a} ^T \rot{\f b} = (\f a \cdot \f b) I - \f b \otimes \f a\, ,
\end{align*}
for all $\f a$, $\f b \in \R^3$, where $I$ denotes the identity matrix in $\R^{3\times 3}$ or 
\begin{align*}
 \quad \rot{\f a} : \nabla \f b = \rot{\f a} : \sk b = \f a \cdot \curl \f b \, , \quad \di \rot{ \f a} = - \curl \f a \, , \quad \frac{1}{2} \rot{\curl \f a} = \sk a \,,
\end{align*}
for all $ \f a, \f b \in \C^1(\ov \Omega)$.
Displaying the cross product by this matrix makes the operation associative. 

Additionally, we define $ \rott{\cdot } : \R^{3 \times 3} \ra \R^3 $, it is the left inverse of $\rot{\cdot} $ and given by
\begin{align*}
\rott{\f A} : = \begin{pmatrix}
\f A_{3,2} \\ \f A_{1,3} \\ \f A_{2,1}
\end{pmatrix}\,,\quad \text{for all } \f A \in \R^{3\times 3}\,.
\end{align*} 
For this mapping holds $ \rott{\rot{\f a}} = \f a $ and, thus $ 2 \rott{\sk{ a}}= \curl \f a$, for all $ \f a \in \C^1(\ov{ \Omega }; \R^3)$. 

We also use the Levi--Civita tensor $\f \Upsilon\in \R^{3\times 3\times 3}$. Let $\mathfrak{S}_3$ be the symmetric group of all permutations of $(1,2,3)$. The sign of  a given permutation  $\sigma \in \mathfrak{S}_3$ is denoted by $\sgn \sigma $.
The Tensor~$\f \Upsilon$ is defined via
\begin{align*}
\f \Upsilon_{ijk}:= \begin{cases}
\sgn{\sigma},  &  ( i ,j,k) = \sigma( 1,2,3)\text{ with } \sigma\in \mathfrak{S}_3 ,\\ 
0, & \text{ else}\, .
\end{cases}
\end{align*}
This tensor allows it two write the cross product as 
\begin{align*}
\f a \times \f b = \f \Upsilon :(\f a \otimes \f b) =\f \Upsilon _{ijk} \f a_j \f b _k\, ,\quad \text{for all }\f a , \f b \in \R^d \, 
\end{align*}
and the curl via
\begin{align*}
\curl \f d = \f \Upsilon : (\nabla \f d)^T =   \f\Upsilon_{ijk} \partial_j \f d_k \,, \quad\text{for all }\f d \in \C^1 ( \Omega)\, .
\end{align*}

For a given Banach space $ V$, Bochner--Lebesgue spaces are denoted  by $ L^p(0,T; V)$. Moreover,  $W^{1,p}(0,T; V)$ denotes the Banach space of abstract functions in $ L^p(0,T; V)$ whose weak time derivative exists and is again in $ L^p(0,T; V)$ (see also
Diestel and Uhl~\cite[Section~II.2]{diestel} or
Roub\'i\v{c}ek~\cite[Section~1.5]{roubicek} for more details).
By  $\C([0,T]; V) $, and $ \C_w([0,T]; V)$, we denote the spaces of abstract functions mapping $[0,T]$ into $V$ that are  continuous and continuous with respect to the weak topology in $V$, respectively.
We often omit the time interval $(0,T)$ and the domain $\Omega$ and just write, e.g., $L^p(\f W^{k,p})$ for brevity.

Finally by $c>0$, we denote a generic positive constant.

\section{Model\label{sec:model}}
\subsection{Governing equations}
Let $ \Omega$ be of class $\C^{3,1}$.
We consider the Ericksen--Leslie model as introduced in~\cite{masswertig}. 
 The governing equations  read as
\begin{subequations}\label{eq:strong}
\begin{align}
\t {\f v}  + ( \f v \cdot \nabla ) \f v + \nabla p + \di \f T^E- \di  \f T^L&= \f g, \label{nav}\\
\f d \times \left (\t {\f d }+ ( \f v \cdot \nabla ) \f d -\sk{v}\f d + \lambda \sy{v} \f d + \f q\right ) & =0,\label{dir}\\
\di \f v & = 0,
\\
| \f d |&=1.
\end{align}%
\end{subequations}

We recall that $\f v : \ov{\Omega}\times [0,T] \ra \R^3$ denotes the velocity  of the fluid, $\f d:\ov{\Omega}\times[0,T]\ra \R^3$ represents the orientation of the rod-like molecules, and $p:\ov{\Omega}\times [0,T] \ra\R$ denotes the pressure.
The Helmholtz free energy potential~$F$, which is described rigorously in the next section, is assumed to depend only on the director and its gradient, $F= F( \f d, \nabla \f d)$.
The free energy functional~$\mathcal{F}$  is defined by
\begin{align*}
\mathcal{F}: \He \ra \R , \quad \mathcal{F}(\f d):= \int_{\Omega} F( \f d, \nabla \f d) \de \f x \,,
\end{align*}
and $\f q$ is its variational derivative (see Furihata and Matsuo~\cite[Section 2.1]{furihata}),
\begin{subequations}\label{abkuerzungen}
\begin{align}\label{qdefq}
\f q :=\frac{\delta \mathcal{F}}{\delta \f d}(\f d) =  \pat{F}{\f d}(\f d , \nabla\f d)-\di \pat{F}{\nabla \f d}(\f d, \nabla \f d)\, .
\end{align}
The Ericksen stress tensor $\f T^E$ is given by
\begin{equation}
\f T^E = \nabla \f d^T \pat{F}{\nabla \f d}( \f d , \nabla\f d ) \, .\label{Erik}
\end{equation}
The Leslie stress tensor is given by
\begin{align}
\begin{split}
\f T^L ={}&  \mu_1 (\f d \cdot \sy{v}\f d )\f d \otimes \f d +\mu_4 \sy{v}
 + {(\mu_5+\mu_6)} \left (  \f d \otimes\sy{v}\f d \right )_{\sym}
\\
& +{(\mu_2+\mu_3)} \left (\f d \otimes \f e  \right )_{\sym}
 +\lambda \left ( \f d \otimes \sy{v}\f d  \right )_{\skw} + \left (\f d \otimes \f e  \right )_{\skw}\, ,
\end{split}\label{Leslie}
\end{align}
where
\begin{align}
\f e : = \t {\f d} + ( \f v \cdot \nabla ) \f d - \sk v\f d\, .\label{e}
\end{align}

We emphasis that Parodi's law is always assumed 
\begin{equation}
 \lambda = \mu_2 + \mu _3\,.\label{parodi}
\end{equation}
It follows from Onsager's reciprocal relation and is essential to prove the energy inequality~\eqref{energyin}. 
\begin{remark}
In our previous work~\cite{weakstrong}, we did not assume Parosi's relation. We rather assumed the energy inequality~\eqref{energyin} to hold even in the absence of Parodi's relation, which we cannot prove. The same strategy could be used in this paper, but for the sake of simplicity, we concentrate on the case of Parodi's relation. In addition, this seems appropriate since the energy inequality is essential for our analysis
\end{remark}

To ensure the dissipative character of the system, we assume that
\begin{align}
\begin{gathered}
 \mu_4 > 0, \quad (\mu_5+\mu_6)- \lambda ^2>0,\quad  \mu_1 >0 
\,.
\end{gathered}\label{con}
\end{align}
\end{subequations}
Finally, we impose boundary and initial conditions as follows:
\begin{subequations}\label{anfang}
\begin{align}
\f v(\f x, 0) &= \f v_0 (\f x) \quad\text{for } \f x \in \Omega ,
&\f v (  \f x, t ) &= \f 0  &\text{for }( t,  \f x ) \in [0,T] \times \partial \Omega , \\
\f d (  \f x, 0 ) & = \f d_0 ( \f x) \quad\text{for } \f x \in \Omega ,
&\f d (  \f x ,t ) & = \f d_1 ( \f x )  &\text{for }( t,  \f x ) \in [0,T] \times \partial \Omega .
\end{align}
\end{subequations}
We always assume that $\f d_1= \f d_0$ on $\partial \Omega$, which is a compatibility condition providing regularity.

\subsection{The general Oseen--Frank energy}
The \textit{Oseen--Frank} energy is given by~(see~Leslie~\cite{leslie}) \
\begin{align*}
F(\f d , \nabla \f d) := \frac{K_1}{2} (\di \f d )^2 +\frac{K_2}{2}( \f d \cdot \curl \f d )^2  + \frac{K_3}{2} |\f d \times \curl \f d|^2 \,,
\end{align*}
where $K_1,K_2,K_3>0$.
This energy can be reformulated using the norm  restriction, $|\f d|=1$, to
\begin{align}
\begin{split}
 F( \f d , \nabla \f d)&:= \frac{k_1}{2} ( \di \f d) ^2 + \frac{ k_2}{2} | \curl \f d |^2 + \frac{k_3}{2} | \f d |^2 ( \di \f d )^2 +   \frac{k_4}{2} ( \f d\cdot \curl \f d )^2  +  \frac{k_5}{2} | \f d \times \curl \f d |^2 \, .
\end{split} \label{frei}
\end{align} 
where $ k_1=k_3=K_1/2$, $k_2={\min\{K_2,K_3\}}/{2}$, $k_4 = K_2-k_2$, and $ k_5 =K_3-k_2$ are again positive constants.
We remark that $| \f d |^2| \curl \f d |^2 = ( \f d \cdot \curl \f d )^2 + | \f d \times \curl \f d |^2 $ for all $\f d \in \C^1 ( \Omega ; \R^3)$.

We introduce short notations for the derivatives of the free energy~\eqref{frei} with respect to $\nabla \f d$ and $ \f d$. The free energy~\eqref{frei} can be seen as a function $F: \R^3 \times \R^{3\times 3}\ra \R $, where we replace $\f d$ in definition~\eqref{frei} by $\f h\in \R^3$ and $\nabla \f d $ by $ \f S\in \R^{3\times 3 }$. By  an easy vector calculation, we find
\begin{align*}
2 F( \f h, \f S)  &=  k_1 \tr (\f S) ^2 + k_2 | ( \f S)_{\skw}|^2  + k_3 | \f h |^2 \tr ( \f S)^2  + k_4 ( \rot{ \f h } : ( \f S) _{\skw})^2 + 
4k_5| (\f S)_{\skw} \f  h|^2 \, ,
\end{align*}
see Section~\ref{sec:not} for the definition of the matrix $\rot{\cdot}$.

We abbreviate the derivative of $F$ with respect to $\f h$ by $F_{\f h}$ and the derivative with respect to $\f S$ by $F_{\f S}$, where
\begin{align*}
F_{\f S} : \R^3 \times \R^{3\times 3 } \ra \R^{3\times 3 } \, , \quad \text{and  } F_{\f h} : \R^3 \times \R^{3\times 3 } \ra \R^3 \, .
\end{align*}
These derivatives are given by
\begin{align}
\begin{split}
F_{\f S}(\f h ,\f S) & =   k _1 \tr (\f S) I + k_2 (\f S)_{\skw}  + k_3 \tr( \f S) | \f h|^2 I  +  k_4  \rot{ \f h}( \rot{\f h} :(\f S)_{\skw})
 + 4 k_5 ((\f S)_{\skw}\f h \otimes \f h ) _{\skw}\,, \\
F_{\f h} ( \f h , \f S) &= k_3 \tr(\f S)^2 \f h + 2 k_4 ( \rot{\f h} :(\f S)_{\skw}) \rott{( \f S)_{\skw}} + 4 k_5 ( \f S)^T_{\skw}(\f S) _{\skw} \f h 
\, ,
\end{split}\label{FSFh}
\end{align}
see Section~\ref{sec:not} for the definition of $\rott{\cdot}$.

To abbreviate, we define the tensor of order 4, $\f \Lambda \in \R^{3^4}$ and a tensor of order 6, $ \f \Theta \in \R^{3^6}$ via
\begin{align}
\f \Lambda_{ijkl} : ={}& k_1 \f \delta_{ij} \f \delta_{kl} + k_2 ( \f \delta_{ik}\f \delta_{jl}-\f\delta_{il}\f\delta_{jk})\,,\label{Lambda}
\intertext{and}
\f \Theta_{ijklmn} :={}&k_3 \f \delta_{ij}\f \delta_{lm}\f \delta_{kn} 
%
%
+ k_5  \left (  \f \delta_{il}\f \delta_{mn}\f \delta_{jk} - \f \delta_{mi}\f \delta_{ln}\f \delta_{jk} - \f \delta_{lj}\f \delta_{mn}\f \delta_{ik} + \f \delta_{jm}\f \delta_{ln}\f \delta_{ik}  \right )\notag
\\
& 
+k_4 \left (  \f \delta_{kn}\f \delta_{jm}\f \delta_{il} + \f \delta_{km}\f \delta_{jl}\f \delta_{in} + \f \delta_{kl}\f \delta_{jn}\f \delta_{im} - \f \delta_{kn}\f \delta_{jl}\f \delta_{im}- \f \delta_{km}\f \delta_{jn}\f \delta_{il} - \f \delta_{kl}\f \delta_{jm}\f \delta_{in}  \right )
\,, \label{ThetaOF}
\end{align}
respectively.
The free energy can be written as 
\begin{align}
2 F(\f d, \nabla \f d ) = \nabla \f d : \f \Lambda : \nabla \f d +  \nabla \f d \otimes  \f d  \dreidots \f \Theta \dreidots  \nabla \f d \otimes \f d  \,. \label{tensoren}
\end{align}
The partial derivatives~\eqref{FSFh} inserted in definition~\eqref{qdefq} 
provides the variational derivative in the case of the Oseen--Frank energy via
\begin{align}
\begin{split}
\f q ={}&- k_1 \nabla \di \f d + k_2 \curl \curl \f d - k_3\nabla (\di \f d | \f d|^2) - k_4 \di \left ( \rot{\f d} ( \f d \cdot \curl \f d ) \right ) - 4 k_5 \di \left (  ( \nabla \f d)_{\skw} \f d \o \f d\right )_{\skw}\\& + k_3 (\di \f d)^2 \f d  + k_4 ( \f d \cdot \curl \f d) \curl \f  d + 4 k_5  ( \nabla \f d)_{\skw}^T ( \nabla \f d)_{\skw}\f d
\\={}&-  \Lap\f d - \di \left ( \f d \cdot \f \Theta \dreidots \nabla \f d \o \f d \right ) + \nabla \f d : \f \Theta \dreidots \nabla \f d \o \f d
 \,.
\end{split}
\label{qdef}
\end{align}
 
The Tensor $\f \Lambda$ is strongly elliptic, i.e.~there is a $\eta>0$ such that $ \f a \otimes  \f b : \f \Lambda : \f a \otimes \f b   \geq \eta | \f a|^2 | \f b|^2 $ for all $\f a, \f b \in \R^3$. Indeed, it holds 
\begin{align}\label{ellip}
 \f a \otimes  \f b : \f \Lambda : \f a \otimes \f b  = k_1 (\f a \cdot \f b)^2 + k_2 ( | \f a |^2 | \f b|^2-( \f a \cdot \f b )^2 ) \geq \min\{k_1,k_2\} | \f a |^2 | \f b|^2\,.
\end{align}
The second order differential operator $-\Lap$ introduced by a strongly elliptic tensor is coercive on $\Hb$, i.e., there exists a constant $k>0$ such that 
\begin{align}\label{kill}
k\| \nabla \f d \|_{\Le}^2 \leq \frac{ 1}{2} \left ( \nabla \f d ; \f \Lambda : \nabla \f d\right )  \,
\end{align}
holds for all $\f d \in \He$.

\section{Generalized solvability concepts and main result \label{sec:rel}}
\subsection{Dissipative solutions}
The concept of dissipative solutions heavily relies on the formulation of an appropriate relative energy for the Oseen--Frank energy. This relative energy serves as a natural comparing tool for two different solutions $(\f v , \f d)$ and $(\vv,\dd)$. 
The relative energy is defined by 
\begin{align}
\begin{split}
\mathcal{E}(t) :
={}& \frac{1}{2} \left \| \f v(t) - \vv(t) \right \| _{\Le}^2 +  \frac{1}{2}\left ( \nabla \f d(t) - \nabla \dd( t) : \f \Lambda : (\nabla \f  d(t)- \nabla \dd( t)) \right ) \\& +\frac{1}{2} \Big ( \nabla\f d (t)\o \f d(t) - \nabla \dd( t) \o \dd( t)  \dreidotkom \f \Theta  \dreidots (\nabla \f d(t) \o \f d (t) 
 - \nabla \dd( t) \o \dd( t) )\Big ) 
  \,
  \end{split}
  \label{relEnergy}
\end{align}
and the relative dissipation by
\begin{align}
\begin{split}
\mathcal{W}(t) : ={}& (\mu_1+\lambda^2) \| \f d(t) \cdot ( \nabla \f v (t))_{\sym} \f d(t) - \dd(t) \cdot ( \nabla \vv(t))_{\sym } \dd (t)\|_{\Le} ^2\\
&+ ( \mu_5 + \mu_6 -\lambda^2 ) \| ( \nabla \f v(t) )_{\sym} \f d(t) - ( \nabla \vv(t))_{\sym} \dd(t)\|_{\Le}^2\\& + \mu_4 \| ( \nabla \f v(t))_{\sym} - ( \nabla \vv(t))_{\sym} \|_{\Le}^2  +  \| \f d(t) \times \f q(t) - \dd(t) \times \tq(t) \|_{\Le}^2 \, . 
\end{split}
\label{relW}
\end{align}
Inserting the definitions of the tensors $\f \Lambda $ and $\f \Theta $,
the relative energy can be expressed via
\begin{align*}
\mathcal{E}(t)={}&  + 
 \frac{k_1}{2}\|\di \f d (t) - \di  \dd( t) \|_{L^2}^2   +  {k_2} \| (\nabla \f d(t))_{\skw} - ( \nabla\dd(t))_{\skw} \|_{\Le}^2 \notag \\&+\frac{k_3}{2}  \| (\di \f d (t)) \f d(t) - (\di \dd(t)) \dd(t) \|_{L^2}^2 + \frac{k_4}{2}\| \f d (t) \cdot \curl\f d(t)  - \dd (t) \cdot \curl\dd(t)  \|_{L^2}^2    \notag
 \\
  &+  2 k_5  \| (\nabla \f d (t))_{\skw} \f d(t) - ( \nabla\dd(t))_{\skw} \dd (t)\|_{\Le}^2 + \frac{1}{2} \left \| \f v(t) - \vv(t) \right \| _{\Le}^2\,.
\end{align*}
We always assume that $(\f v , \f d)$ and $(\vv,\dd)$ fulfill the regularity requirements~\eqref{reldiss} below such that~\eqref{relEnergy} and~\eqref{relW} are well defined. 
In the following $(\vv, \dd)$ are even assumed to fulfill~\eqref{regtest} below.
%

\begin{definition}[Dissipative solution]\label{def:diss}
The triple $(\f v , \f d , \f q)$ consisting of the velocity field $\f v $, the director field $\f d$ and the variational derivative $\f q$ is  said to be a dissipative  solution to~\eqref{eq:strong}--\eqref{frei} if
\begin{subequations}\label{reldiss}
\begin{align}
\f v &\in \C_w(0,T;\Ha)\cap  L^2(0,T;\V)\,,
\\ \f d& \in \C_w(0,T;\He) \text{ with } | \f d (\f x ,t)| =1 \text{ a.\,e.~in $\Omega\times (0,T)$}\,,\\
\f d \times \f q & \in L^2(0,T; \Le)\, 
\end{align}
\end{subequations}
and if 
\begin{align}
\begin{split}
\frac{1}{2}\mathcal{E} (t) + {}&\frac{1}{2}\int_0^t\mathcal{W}(s) \exp\left ({\int_s^t\mathcal{K}(\tau)\de \tau }\right )\de s\\ \leq{}&  \Big ( \mathcal{E}(0)+\frac{|\f \Theta|^2}{2 k}\|\nabla \dd \o \dd \|_{L^\infty (L^\infty)} ^2 \left \|   \f d(0)- \dd(0)    \right \|_{\Le}^2 \Big) \exp\left ({\int_0^t\mathcal{K}(s)\de s } \right )
\\
& +\Big ((\nabla \f d(0)-\nabla \dd (0))\o (\f d(0)-  \dd(0)) \dreidotkom \f \Theta \dreidots \nabla \dd (0)\o \dd(0) \Big) \exp\left ({\int_0^t\mathcal{K}(s)\de s }\right ) 
\\&
+ \int_0^t \left ( \mathcal{A}(s) , 
\begin{pmatrix}
\vv -\f v \\
\f d \times (\tq -   \f q  )  +  \frac{|\f \Theta|^2}{k}\|\nabla \dd \o \dd \|_{L^\infty(\f L^\infty)} ^2   \f d   \times \dd
\end{pmatrix}\right )
\exp\left ({\int_s^t\mathcal{K}(\tau)\de \tau }\right )\de s  
\, 
\end{split}\label{relenin}
\end{align} 
for all  test functions $(\vv , \dd)$ with  
\begin{align}
\begin{split}
\vv &\in L^\infty(0,T; \Ha)\cap L^2(0,T; \f L^\infty)\cap  L^2 (0,T; \f W^{1,3}_{0,\sigma})\cap L^1(0,T; \f W^{1,\infty})\cap W^{1,2}(0,T;\Vd)\\
\dd &\in    L^\infty(0,T;\f W^{1,\infty
}) \cap  L^2( 0,T ; \f W^{2,3}) \cap  L^4( 0,T ; \f W^{1,6})\cap W^{1,1 }(0,T;  \f W^{1,3}\cap \f L^\infty ) \text{ with }| \dd| =1 \text{ a.e. in }\Omega\times (0,T)\,.
\end{split}\label{regtest}
\end{align}
as well as  
\begin{align}
&\intte{\left ( \f d (t) \times\left ( \t \f d(t)+( \f v(t)\cdot \nabla ) \f d(t) -  \sk{v(t)}\f d(t) + \lambda \sy{v(t)} + \f q(t)\right ), \f \zeta(t)\right )}=0
\label{eq:mdir}
\end{align}
for $\f \zeta \in L^2(0,T; \f L^3)$   and 
$\dd$ fulfilling inhomogeneous boundary conditions such that $\tr(\dd)=\f d _1 $ on $\partial \Omega$.
%

The potential $\mathcal{K}$ is given by 
\begin{align*}
\mathcal{K}(s) ={}& C  \left (  \|\vv(s)\|_{\f L^\infty}^2+ \| \vv(s)\|_{\f W^{1,3}}^2 + \| \nabla^2 \dd\|_{\f L^{3}}^2 + \| \nabla \dd (s)\|_{\f L^{6}}^4 + \| \t \dd(s)\|_{\f L^\infty} + \| \t \dd(s)\|_{\f W^{1,3}} \right ) \\ & +C  \left ( \|(\nabla\vv(s))_{\sym}\|_{\f L^\infty}+\| \nabla \dd \o \dd \|_{L^\infty(\f L^\infty)}^4 \|\f d (s) \times \dd ( s)\|_{\f L^3}^2 +1  \right )\,,
\end{align*}
where $C$ is a possible large constant depending on the norms $\| \f d \|_{L^\infty(\f L^\infty)} $, $\| \dd \|_{L^\infty(\f L^{\infty})}$, which are just $1$ and the norm $\|\nabla \dd \|_{L^\infty (\f L^3)}$. The constant $k$ is given in~\eqref{kill}. It is obvious that $\mathcal{K}$ is bounded in $L^1(0,T)$ due to the regularity assumptions~\eqref{regtest}. 


The operator $\mathcal{A}$ incorporates the classical formulation~\eqref{eq:strong} evaluated at the test functions $(\vv ,\dd)$
\begin{align}
\mathcal{A}(t) = \begin{pmatrix}
 \t {\vv(t)}  + ( \vv(t) \cdot \nabla ) \vv(t) + \di \tilde{\f T}^E(t)- \di  \tilde{\f T}^L(t)- \f g(t)\\
\dd(t) \times \left (\t {\dd (t)}+ ( \vv(t) \cdot \nabla ) \dd(t) -(\nabla \dd(t))_{\skw} \dd(t) + \lambda (\nabla \dd(t))_{\sym} \dd(t) + \tq(t)\right ) 
\end{pmatrix}\label{A}
\end{align}
\end{definition}
\begin{remark}[Weak-strong uniqueness]\label{rem:weakstrong}
 The weak-strong uniqueness of dissipative solutions immediately follows from the definition. Indeed, if there exists a classical solution $(\vv ,\dd)$ fulfilling the initial conditions and equations in a strong sense and admitting the regularity~\eqref{regtest}, then it must be the unique solution. 
 Since $\mathcal{A}$ vanishes for a strong solution, the right-hand side of~\eqref{relenin} vanishes for all functions fulfilling the same initial conditions. 
 Note that it is sufficient if $(\vv ,\dd)$ is a weak solution of~\eqref{eq:strong} such that the equations~\eqref{nav} and \eqref{dir} are fulfilled in $L^2(0,T;\Vd)$ and $L^2(0,T;\Le)$, respectively. Then the formulation of the relative energy inequality~\eqref{relenin} especially~\eqref{A} is already well defined. 
\end{remark}
\begin{remark}[Stability of the relative energy inequality]
The inequality~\eqref{relenin} is stable under the convergence in the spaces~\eqref{reldiss}. 
Indeed, consider a sequence of functions $\{(\fn v , \fn d ,\fn d \times \fn q)\}$ converging to $(\f v ,\f d, \f d \times \f q)$, with
\begin{subequations}\label{remconv}
\begin{align}
\fn v & \ra \f v&\text{ in }  \C_w(0,T;\Ha) &&&&\qquad &\fn v \rightharpoonup \f v & &\text{ in } L^2(0,T;\V)\\
\fn d & \ra \f d & \text{ in } \C_w(0,T;\He)&&&& \qquad &\fn d \times \fn q \rightharpoonup \f d \times \f q &&\text{ in } L^2(0,T;\Le)
\end{align}
\end{subequations}
and $| \fn d(\f x ,t)  | = 1 $ for a.e. $(\f x ,t) \in \Omega\times (0,T)$.  
Due to the weakly-lower semi-continuity of the appearing norms, we may infer $\| \f v (t)\|_{\Le}^2 \leq \liminf_{n\in \N} \| \fn v(t)\|_{\Le}^2 $ as well as $ \left ( \nabla \f d (t) ; \f \Lambda : \nabla \f d(t) \right ) \leq \liminf_{n\in \N} \left ( \nabla \fn d (t) ; \f \Lambda : \nabla \fn d(t) \right )$. 
With $\fn d \ra \f d$ in $  \C_w(0,T;\He)$ and the compact embedding in three dimensions, there exists a strongly converging subsequence such that $\fn d \ra \f d $ in $ \C([0,T];\f L^5)$ implying that $ \f \Theta  \dreidots \nabla \fn d \o \fn d \ra \f \Theta  \dreidots \nabla \f d \o \f d$ in $\C_w([0,T]; \f L^{\nicefrac{10}{7}})$ with the boundedness in $L^\infty(0,T; \Le)$ one obtains with~\cite[Page~297]{magnes} even the convergence in $\C_w([0,T]; \Le)$. Together we observe that $\mathcal{E}$ is weakly-lower semi-continuous with respect to the convergence in~\eqref{reldiss}. Due to the boundedness of the director and the weak convergences on the right-hand side of~\eqref{remconv}, the realtive dissipation~\eqref{relW} is also weakly-lower semi-continuous.  
\end{remark}
%

\begin{remark}[Other dissipative solutions]\label{rem:feireisl}
The concept of dissipative solutions for the Navier--Stokes--Fourier system introduced by Feireisl~\cite{feireislsingular} differs from the solution concept of Definition~\ref{def:diss}. 
To get formulation~\eqref{relenin}, we already applied Gronwall's estimate, where in the case of the relative energy inequality for the Navier--Stokes--Fourier system, the right-hand side remains untouched and  Gronwall's argument is applied during the proof of the weak-strong uniqueness~\cite{novotny}.  

In the case of the Ericksne--Leslie system equipped with the Oseen--Frank energy, this would still lead to a  measure-valued formulation.   

Note that a solvability concept relying on an inequality is well known in the context of Gradient flows (see for instance~\cite[Proposition~23.1]{opttrans}). For a vanishing velocity field, the Gradient flow of the Oseen--Frank energy can be formulated by an upper energy dissipation estimates (see~\cite[Theorem~3.2]{mielke}) which corresponds to the dissipative formulation\eqref{def:diss} for $\f v = \vv = \dd = 0$.
Indeed, with $\f v \equiv 0$, we may infer from~\eqref{eq:mdir} that $\f d \times \t \f d + \f d \times \f q = 0$. The relative energy inequality~\eqref{relenin} simplifies to 
\begin{align}
\mathcal{F}(\f d(t) ) + \int_0^t \left (\frac{1}{2} \| \f d \times \t \f d \|_{\Le}^2 + \frac{1}{2} \| \f d \times \f q \|_{\Le}^2\right )  \de s = \mathcal{F}(\f d_0) \label{uneq}
\end{align}
where we used that $1/2  \| \f d \times \f q \|_{\Le}^2 = 1/2 \| \f d \times \t \f d \|_{\Le}^2 $.  The inequality~\eqref{uneq} corresponds to the upper energy dissipation estimate (see~\cite[Theorem~3.2]{mielke}). 
\end{remark}

In the next section, we introduce the concept of measure-valued solutions. The proof of global existence of measure-valued solutions to the Ericksen--Leslie model equipped with the Oseen--Frank energy is executed by the author in~\cite{masswertig} and the weak-strong uniqueness of these solutions is proven in~\cite{weakstrong}. We also refer to~\cite{masswertig} for a more extensive introduction into the concept of generalized gradient Young measures.
\subsection{Measure-valued solutions}
\begin{definition}[measure-valued solutions]\label{def:meas}
The tupel $( ( \f v ,\f d ) , ( \nu^o,m , \nu^\infty) , ( \mu , \nu^\mu))$ consisting of the pair $(\f v , \f d)$ of velocity field $\f v$ and director field  $\f d$, the generalized gradient Young measure $( \nu^o,m , \nu^\infty) $ and the defect measure $(\mu , \nu^\mu)$ (see below)  is  said to be a measure-valued solution to~\eqref{eq:strong} if
\begin{align}
\begin{split}
\f v &\in L^\infty(0,T;\Ha)\cap  L^2(0,T;\V)
\cap W^{1,2}(0,T; ( \f W^{1,3}_{0,\sigma})^*),
\\ \f d& \in L^\infty(0,T;\He)\cap   W^{1,2}  (0,T;  \f  L^{\nicefrac{3}{2}} ),\\
\{\nu^o _{(\f x,t)}\}&  \subset \mathcal{P} ( \R^{3\times 3})\,, \text{ a.\,e.~in $\Omega\times (0,T)$} \, ,\\
 \{m_t\} &\subset \mathcal{M}^+(\ov\Omega)\,,\text{ a.\,e.~in $ (0,T)$} \, ,  \\
 \{\nu^\infty _{(\f x,t)}\} &\subset \mathcal{P}(\ov B_3\times \Se^{3^2-1})\,,\text{ $m_t $-a.\,e.~in }\ov\Omega \text{ and a.\,e.~in }   (0,T)\, ,\\
 \{\mu_t\} &\subset \mathcal{M}^+(\ov\Omega)\,,\text{ a.\,e.~in $ (0,T)$} \, ,  \\
\{ \nu^{\mu}_{(\f x ,t)}\}&\subset \mathcal{P}(\Se^{3^3-1})\,, \text{ $\mu_t$-a.\,e.~in }\ov\Omega \text{ and a.\,e.~in }   (0,T)\, 
\end{split}\label{measreg}
\end{align}
and if
\begin{subequations}\label{meas}
\begin{align}
\begin{split}
\int_0^T (\t \f v(t), \f \varphi(t)) \de t + \int_0^T ((\f v(t)\cdot \nabla) \f v(t), \f \varphi(t)) \de t  - \intte{ \ll{\nu_t,\f S^T F_{\f S}(\f h, \f S):\nabla \f \varphi(t)   }  }&
\\-2 \int_0^T \ll{\mu_t, \f \Gamma \dreidots (\f \Gamma\cdot \nabla\f \varphi(t))}\de t + \intte{(\f T^L(t): \nabla \f \varphi(t) ) } &={}\intte{ \left \langle \f g (t),\f \varphi(t)\right \rangle }\, ,
\end{split}\label{eq:velo}
\intertext{as well as~\eqref{eq:mdir} with }
\begin{split}
 \intte{ \ll{\nu_t, \f \Upsilon :\left (\f  S    (F_{\f S}(\f h, \f S))^T\right ) \cdot\f \psi(t)   }}+\intte{ \ll{\nu_t, \left (\f h \times F_{\f h}(\f h, \f S)\right ) \cdot\f \psi(t)   }}&
\\
+ \intte{(\rot{\f d(t)}  F_{ \f S} ( \f d(t) , \nabla \f d (t)) ; \nabla \f \psi(t) )}& 
 ={}\intte{\left ( \f d (t) \times \f q(t) , \f \psi(t)\right )}  \,, 
 \end{split}
\label{eq:q}
\end{align}%
\end{subequations}
holds for all $ \f \varphi \in \mathcal{C}_c^\infty(\Omega \times ( 0,T);\R^3))$ with $ \di \f \varphi =0$, $\f \zeta \in L^2(0,T;\Le)$ and $ \f \psi \in  \mathcal{C}_c^\infty(\Omega \times ( 0,T);\R^3))$, respectively.
Additionally, the norm restriction of the director holds, i.\,e. $|\f d (\f x,t)|=1$ for a.\,e.~$(\f x, t)\in \Omega\times (0,T)$, the oscillation measure of a linear function is the gradient of the director
\begin{align}
  \int_{\R^{3\times 3} } \f S  \nu^o_{(\f x, t)} ( \de \f S)   = \nabla \f d(\f x,t) \, \quad \text{for a.e. } (\f x ,t) \in \Omega\times (0,T)\,,\label{identify}
\end{align}
and the
initial conditions $( \f v_0, \f d_0)\in \Ha \times \Hc$ with $ \f d_0 \in \Hrand{7}$ shall be fulfilled in the weak sense and the boundary conditions in the sense of the trace.
We remark that the trace is well defined for the function $\f d \in L^\infty(0,T;\He)$, which is the expected value of the oscillation measure $\nu^o$. 

The dual pairings are defined as
\begin{align*}
 \ll{\mu_t,f} :={}& \int_{\ov \Omega} \int_{\Se ^{3^3-1}}  f(\f x,t,\f \Gamma) \nu^\mu_{(\f x ,t)} ( \de \f \Gamma) \mu_t(\de \f x )\,\intertext{ for $f \in \C(\Se^{3^3-1};\R) $  and}
\ll{\nu_t, f } :={}& 
 \int_{\Omega} \int_{\R^{3\times 3} } f(\f x, t,\f d(\f x, t), \f S)  \nu^o_{(\f x, t)} ( \de \f S)\de \f x 
 + \int_{\ov\Omega}\int_{\Se^{3^2-1} \times \ov B_3} \tilde{f} (\f x, t, \tilde{\f h} , \tilde{\f S}) \nu_{(\f x, t)}^\infty ( \de \tilde{\f S}, \de \tilde{\f h}) m_t (\de \f x)
\end{align*}
for $f\in \mathcal{R}$ (see~\eqref{transi} below).
\end{definition}
We refer to the section~\ref{sec:not} for the definition of the tensor $\f \Upsilon$ and to \eqref{transi} for the definition of the transformed function~$\tilde{f}$.
\begin{remark}
We often abuse  the notation by writing $ \ll{\nu_t , f(\f h , \f S)}$. Thereby, we mean the generalized Young measure applied to the continuous function $(\f h, \f S)\mapsto f (\f h ,\f S)$. 
\end{remark}
The transformed function $\tilde{f}:\ov \Omega \times [0,T]   \times  B_3 \times B_{3 \times 3}\ra \R$, the so-called recession function is given by
\begin{align}
\tilde{f} ( {\f x },t, \tilde{\f h} ,\tilde{\f  S} ) : = 
f ( \f x ,t, \frac{\tilde{\f h}}{\sqrt{1-|\tilde{\f h}|^2}}, \frac{\tilde{\f S}}{\sqrt{1-|\tilde{\f S}|^2}}))  ( 1-|\tilde{ \f h}|^2  )( 1- | \tilde{\f S}|^2) \,. \label{transi}
\end{align}
The class of functions for which the above representation is valid are those functions, $f\in \C(\ov \Omega  \times \R^3 \times \R^{3\times 3})$ such that $\tilde{f}$ admits a continuous extension on the closure of its domain
\begin{align*}
\mathcal{F}:= \Big \{ f \in \C ( \ov \Omega \times [0,T] \times \R^3 \times \R^{3\times 3 }  ) |& \exists \tilde{ g}\in \C(\ov \Omega \times [0,T] \times \ov B_3\times \ov B_{3\times 3}\, ;
\, \tilde{f}= \tilde{g}\text{ on }\ov \Omega \times [0,T] \times B_d \times B_{d\times d}  	\Big \}\,.
\end{align*} 
In comparison to weak solutions (see~\cite{unsere}) the Ericksen stress $\f T^E$ and the variational derivative $\f d \times \f q$ are in this measure-valued formulation represented by generalized Young measures. 
A \textbf{generalized Young measure} on $\ov \Omega \times [0,T]$ with values in $\R^d \times \R^{d\times d}$ is a triple $( \nu_{(\f x,t)}^o , m_t , \nu_{(\f x,t)}^\infty) $ consisting of  
\begin{itemize}
\item a parametrized family of probability measures $\{\nu_{(\f x,t)}^o\} \subset \mathcal{P}(\R^{d\times d})$ for a.\,e. $(\f x,t)  \in \Omega \times (0,T)$,
\item a positive measure $\{m_t\}\subset \M^+(\ov \Omega)$, for a.\,e. $t\in (0,T)$ and
\item a parametrized family of probability measures $\{ \nu_{(\f x,t)}^\infty\}\subset \mathcal{P}( \ov B_d \times \Se^{d^2-1})$, for $m_t$-a.\,e.~$\f x  \in \ov\Omega $ and a.\,e.~$t\in(0,T)$. 
\end{itemize}
As in~\cite[page 552]{rindler} we call $\nu^o$ \textit{oscillation measure}, $m_t$ \textit{concentration measure} and $\nu^\infty$ the \textit{concentration angle measure}. 

%
In the case of the Ericksen stress, an additional defect measure is of need to describe the limit of the regularised system we considered in~\cite{unsere}.
A \textbf{defect measure} on $\ov \Omega \times [0,T]$ with values in $ \R^{d\times d\times d }$
is a pair $(\mu_t, \nu^\mu )$ consisting of 
\begin{itemize}
\item a positive measure $\mu_t\in\M^+(\ov\Omega)$, for a.\,e. $t\in (0,T)$ and
\item a parametrized family of probability measures $\{ \nu_y^\mu\}_{y\in \ov Q} \in \mathcal{P}(  \Se^{d^3-1})$, for $\mu_t$-a.\,e.~$\f x \in \ov \Omega$ and a.\,e.~$t\in(0,T)$. 
\end{itemize}
We refer again to~\cite{masswertig} for more details on the convergence in the sense of generalized Young measures.

\begin{definition}[Suitable measure-valued solutions]
A measure-valued solution is said to be a suitable measure-valued solution if it fulfills Definition~\ref{def:meas} and additionally the energy inequality
\begin{align}
\begin{split}
 &\frac{1}{2}\|\f v (t)\|_{\Le}^2 + \ll{\nu_t, F} + \frac{1}{2}\ll{\mu_t , 1 }   + \inttet{(\mu_1+\lambda^2)\|\f d\cdot \sy{v}\f d\|_{L^2}^2 }  \\
& +\inttet{ \left [
  \mu_4 \|\sy{v}\|_{\Le}^2+( \mu_5+\mu_6-\lambda^2)\|\sy{v}\f d\|_{\Le}^2  +  \|\f d \times \f q\|_{\Le}^2\right ]}
\\
& \qquad\qquad \leq  \left ( \frac{1}{2}\|\f v_0 \|_{\Le}^2 + \F( \f d_0)\right )
 + \inttet{\left [\langle \f g , \f v \rangle 
 \right ]}\, .
\end{split}
\label{energyin}
\end{align}
a.e. in $(0,T)$.

\end{definition}
\begin{theorem}\label{thm:main}
 Let $((\f v , \f d), ( \nu^o,m , \nu^\infty) , ( \mu , \nu^\mu))$ be a suitable measure-valued solution to the Ericksen--Leslie model according to Definition~\ref{def:meas}.
Then the solutions $(\f v , \f d)$, with $\nabla \f d(\f x ,t) = \mathbb{E}(\nu^o_{(\f x ,t)})= \int_{\R^{3\times 3}} \f S \nu^o_{(\f x ,t ) } (\de \f S) $ and $\f d \times  \f q $ given in~\eqref{eq:q} is a dissipative solution to the Ericksen--Leslie model according to Definition~\ref{def:diss}.
 \end{theorem}

\begin{remark}
The theorem grants the global existence of dissipative solutions, since the global existence of measure-valued solutions was proven in~\cite{masswertig} for $\f g \in L^2(0,T; \Vd)$. The weak-strong uniqueness of dissipative solutions is intrinsically fulfilled (see Remark~\ref{rem:weakstrong}).
\end{remark}

The proof of this main result is very similar to the weak-strong uniqueness proof in~\cite{weakstrong}.
In comparison to the proof in~\cite{weakstrong}, we do not assume that $(\vv, \dd)$ is a strong solution. These are only assumed to be appropriate test functions (see Definition~\ref{def:diss}).

 \section{Proof of the main result\label{sec:pre}}
  \begin{proposition}[Shifted energy equality]\label{prop:shift}
 Let $(\vv,\dd)$ be a test function fulfilling the regularity requirements~\eqref{regtest}. 
 Then it fulfills the following shifted energy equality
 \begin{align*}
\begin{split}
 &\frac{1}{2}\|\vv (t)\|_{\Le}^2 + \mathcal{F}(\dd(t))   + \inttet{\left [(\mu_1+\lambda^2)\|\dd(s)\cdot (\nabla \vv (s))_{\sym}\dd(s)\|_{L^2}^2 +
  \mu_4 \|(\nabla \vv (s))_{\sym}\|_{\Le}^2 \right ]}  \\
& +\inttet{\left [( \mu_5+\mu_6-\lambda^2)\|(\nabla \vv (s))_{\sym}\dd(s)\|_{\Le}^2  +  \|\dd(s) \times \tq(s)\|_{\Le}^2\right ]}
\\
& \qquad =  \left ( \frac{1}{2}\|\vv(0) \|_{\Le}^2 + \F( \dd(0))\right )
 + \inttet{\left [\langle \f g(s) , \vv(s) \rangle + \left (\mathcal{A}(s), \begin{pmatrix}
 \vv (s)\\ \dd (s)\times \tq(s)
 \end{pmatrix}\right )
 \right ]}\, ,
\end{split}
\end{align*}
for all $t\in[0,T]$, 
where $\mathcal{A}$ is given in~\eqref{A}. 
 \end{proposition}
 \begin{proof}
 The proof is similar to the proof of~Proposition~3.1 in~\cite{unsere} despite the fact that we assume that $(\vv,\dd)$ are test functions and not solutions of~\eqref{eq:strong}. 
  Equation~\eqref{nav} and equation~\eqref{dir} evaluated at $(\vv, \dd)$ are tested with $\vv$ and $\dd \times \tq $, respectively. Then both equations are summed up. 
This gives the last term on the right hand side.
 On the left hand side, we use the same manipulations as in~\cite[Proposition~3.1]{weakstrong}. 
Remark that the identities  $ \rot{\dd}^T \rot{\dd} = I - \dd \o \dd $ , $ \t | \dd|^2=0 $ as well as $  (\vv \cdot\nabla)| \dd|^2=0$ hold for the unit vector $\dd$. 
For the Leslie-stress~$\f T^L$ tested with $\syv$, we use \cite[Corollary~3.1]{weakstrong} stating that $\tilde{\f e } = - \rot{\dd}^T \rot{\dd} ( \lambda \syv \dd + \tq)$. This leads to 
\begin{align*}
  (\tilde{\f T}^L& ; \syv )- ( \dd \times \skv \dd , \dd \times \tq )+ \lambda ( \dd \times \syv \dd , \dd \times \tq ) \\
  ={}& \mu_1 | \dd \cdot \syv \dd|^2 +  \mu_4 |\syv|^2 + ( \mu_5 + \mu_6) | \syv \dd|^2 + \lambda (\te ,\syv \dd) 
  \\&+ \lambda ( \syv \dd , \skv \dd) + ( \te , \skv \dd) - ( \dd \times \skv \dd , \dd \times \tq )+ \lambda ( \dd \times \syv \dd , \dd \times \tq ) \\
  ={}&(\mu_1+ \lambda^2) | \dd \cdot \syv \dd|^2 +  \mu_4 |\syv|^2 + ( \mu_5 + \mu_6- \lambda^2) | \syv \dd|^2 \\
  &-  \lambda  (\dd \times \tq  ,\dd \times \syv \dd) + \lambda ( \dd \times \syv \dd , \dd \times \tq ) 
  + ( \dd \times \skv \dd , \dd \times \tq ) -( \dd \times \skv \dd , \dd \times \tq )
   \\={}&(\mu_1+ \lambda^2) | \dd \cdot \syv \dd|^2 +  \mu_4 |\syv|^2 + ( \mu_5 + \mu_6- \lambda^2) | \syv \dd|^2 
  \,.
\end{align*}
Summing up both tested equations and integrating in time gives the desired shifted energy equality.

 \end{proof}


\subsection{Relative measure-valued energy}
The relative measure-valued energy is defined by 
\begin{align}
\begin{split}
\mathcal{E}^M(t) :
={}&\frac{1}{2} \ll{\mu,1}+ \frac{1}{2} \left \| \f v(t) - \vv(t) \right \| _{\Le}^2 +  \frac{1}{2}\ll{\nu_t ,(\f S - \nabla \dd( t) ) : \f \Lambda : (\f  S- \nabla \dd( t)) } \\& +\frac{1}{2} \ll{\nu_t ,(\f S\o \f h - \nabla \dd( t) \o \dd( t) ) \dreidots \f \Theta  \dreidots (\f S\o \f h  
 - \nabla \dd( t) \o \dd( t) )}
  \,.
  \end{split}
  \label{relEn}
\end{align}
Inserting the definitions of the tensors $\f \Lambda $ and $\f \Theta $,
the relative measure-valued energy can be expressed as
\begin{align*}
\mathcal{E}^M(t)={}& \frac{1}{2}\int_\Omega \mu_t(\de \f x) + 
 \frac{1}{2}\ll{\nu_t ,k_1(\tr{(\f S)} - \tr{(\nabla \dd( t) )})^2   + 2 k_2 | (\f S)_{\skw} - ( \nabla\dd(t))_{\skw} |^2 }\notag \\&+\frac{1}{2} \ll{\nu_t, k_3| \tr(\f S) \f h - (\di \dd(t)) \dd(t) |^2 + k_4((\f S)_{\skw}:  [ \f h]_{\f X} - ( \nabla \dd(t))_{\skw}:  [ \dd(t) ]_{\f X})^2    }\notag
 \\
  &+2 k_5\ll{\nu_t,    | (\f S)_{\skw} \f h - ( \nabla\dd(t))_{\skw} \dd (t)|^2   } + \frac{1}{2} \left \| \f v(t) - \vv(t) \right \| _{\Le}^2\,.
\end{align*}
We remark that due to the regularity shown in~\cite{masswertig}, it holds $\mathcal{E}^M\in L^\infty(0,T)$. 

\begin{proposition}\label{prop:relenab}
Let $((\f v ,\f d ),(\nu^o,m,\nu^\infty),(\mu,\nu^\mu))$ be a measure valued solution and $(\vv, \dd)$ fulfilling~\eqref{regtest}.
Then 
\begin{align*}
\mathcal{E}(t)\leq \mathcal{E}^M(t) \, 
\end{align*}
for a.e.~$t\in(0,T)$ with $\nabla \f d (\f x ,t) =  \mathbb{E}(\nu^o_{(\f x ,t)})= \int_{\R^{3\times 3}} \f S \nu^o_{(\f x ,t ) } (\de \f S) $.

\end{proposition}
\begin{proof}
The main idea to prove the assertion is to apply Jensen's inequality on the probability measure~$\nu^o$. 

First, we observe that $\mu$ is a positive measure such that $\ll{\mu_t,1}\geq 0$. 
Similarly, the defect measure $\nu^\infty$ can be estimated by zero. Indeed, using the definition of the generalized gradient Young measure, we observe that 
\begin{align*}
\ll{\nu_t ,(\f S - \nabla \dd( t) ) : \f \Lambda : (\f  S- \nabla \dd( t)) } = {}& \int_{\Omega} \int_{\R^{d\times d} }(\f S - \nabla \dd( t) ) : \f \Lambda : (\f  S- \nabla \dd( t))  \nu^o_{(\f x, t)} ( \de \f S)\de \f x \\&+ \int_{\ov\Omega}\int_{\Se^{d^2-1} \times \ov B_d} \tilde{\f S} : \f \Lambda : \tilde{\f S} ( 1- | \tilde{\f h}|^2) \nu_{(\f x, t)}^\infty ( \de \tilde{\f S}, \de \tilde{\f h}) m_t (\de \f x)\\
\geq {}& \int_{\Omega} \int_{\R^{d\times d} }(\f S - \nabla \dd( t) ) : \f \Lambda : (\f  S- \nabla \dd( t))  \nu^o_{(\f x, t)} ( \de \f S)\de \f x \,,
\end{align*}
since $\nu^\infty$ and $m_t$ are positive measures and the integrand is also positive. 
We observe that $ \f S \mapsto \f S : \f \Lambda : \f S $ is a quasi-convex function (see~\cite{quasiconvex}).
A function $f: \R^{3\times 3} \ra \R$ is called quasi-convex, if 
\begin{align}\label{quasiconvex}
f( \f A) \leq \frac{1}{|D |} \int_{D} f( \f A + \nabla \f  \xi(\f y) ) \de \f y  \qquad\text{for all } \f \xi \in \C_c^\infty(D;\R^3) \text{ and }\f A \in \R^{3\times 3}\,.
\end{align}
For the function $f(\f A) = \f A : \f \Lambda : \f A $ it holds
\begin{align*}
\int_{D} ( \f A + \nabla  \f  \xi(\f y)  ) : \f \Lambda : ( \f A + \nabla \f  \xi(\f y)  ) \de \f y   - \int_{D}  \f A : \f \Lambda : \f A  \de \f y = 2 \int_{D} \di (   (\f \Lambda :  \f  A)^T \f \xi (\f y) )   \de \f y + \int_{D}  \nabla  \f  \xi(\f y)   : \f \Lambda :  \nabla \f  \xi(\f y)   \de \f y \geq 0 \,.
\end{align*}
The first term on the right-hand side vanishes with the divergence theorem and $\f \xi \in  \C_c^\infty(D;\R^3)$ and the last term on the right-hand side is positive. 
The definition~\eqref{quasiconvex} is fulfilled and with Jensen's inequality for quasi-convex functions~\cite[Theorem~1.1]{KinderlehrerPedregal}, we find
\begin{align*}
 \int_{\R^{d\times d} }(\f S - \nabla \dd(\f x, t) ) : \f \Lambda : (\f  S- \nabla \dd(\f x, t))  \nu^o_{(\f x, t)} ( \de \f S)  \geq  ( \nabla \f d(\f x ,t) - \nabla \dd(\f x ,t)) : \f \Lambda :( \nabla \f d(\f x ,t) - \nabla \dd(\f x ,t))\,.
\end{align*}
Note that the equality~\eqref{identify} holds.

The same argumentation is applied to the term in the second line of~\eqref{relEn}.
The defect-measure $\nu^\infty$ is a positive measure such that 
\begin{align*}
&\ll{\nu_t ,(\f S\o \f h - \nabla \dd( t) \o \dd( t) ) \dreidots \f \Theta  \dreidots (\f S\o \f h  
 - \nabla \dd( t) \o \dd( t) )} \\ {}&\qquad= \int_{\Omega} \int_{\R^{d\times d} }(\f S\o \f d(\f x ,t) - \nabla \dd( \f x ,t) \o \dd (\f x ,t)) \dreidots \f \Theta \dreidots (\f  S\o \f d(\f x ,t)- \nabla \dd( \f x ,t)\o \dd(\f x ,t))  \nu^o_{(\f x, t)} ( \de \f S)\de \f x \\&\qquad\quad+ \int_{\ov \Omega }\int_{\Se^{d^2-1} \times \ov B_d} \tilde{\f S}\o \tilde{\f h} \dreidots \f \Theta \dreidots \tilde{\f S}\o \tilde{\f h}  \nu_{(\f x, t)}^\infty ( \de \tilde{\f S}, \de \tilde{\f h}) m_t (\de \f x)\\
&\qquad\geq {} \int_{\Omega} \int_{\R^{d\times d} }(\f S\o \f d(\f x ,t) - \nabla \dd( \f x ,t) \o \dd (\f x ,t)) \dreidots \f \Theta \dreidots (\f  S\o \f d(\f x ,t)- \nabla \dd( \f x ,t)\o \dd(\f x ,t)) \nu^o_{(\f x, t)} ( \de \f S)\de \f x \,.
\end{align*}
The function~$f(\f A) = \f A \o \f d  \dreidots \f \Theta \dreidots \f A \o \f d  $ is again quasi-convex
\begin{align*}
\int_{D} ( \f A + \nabla  \f  \xi(\f y)  )& \o \f d  \dreidots \f \Theta \dreidots( \f A + \nabla \f  \xi(\f y)  ) \o \f d  \de \f y  - \int_{D}   \f A \o \f d  \dreidots \f \Theta \dreidots \f A \o \f d  \de \f y \\={}& 2 \int_{D} \di (  (\f d  \cdot \f \Theta \dreidots \f A \o \f d)\cdot \f \xi (\f y)  )   \de \f y + \int_{D}  \nabla  \f  \xi(\f y)   \o \f d  \dreidots \f \Theta \dreidots   \nabla \f  \xi(\f y) \o \f d   \de \f y \geq{}0\,.
\end{align*}
Note that the first term on the right-hand side of the equality-sign is zero since $\xi \in \C_c^\infty(D;\R^3)$ and the last term is non-negative. We investigated the quasiconvexity in every point $(\f x ,t)$ such that $\f d $ is a constant vector. The linearity of $ \f \xi$ in the second term on the right-hand side is essential.  
Jensen's inequality for quasi-convex functions~\cite[Theorem~1.1]{KinderlehrerPedregal} yields
\begin{multline*}
\int_{\R^{d\times d} }(\f S\o \f d(\f x ,t) - \nabla \dd(\f x , t) \o \dd (\f x ,t)) \dreidots   \f \Theta \dreidots (\f  S\o \f d(\f x ,t)- \nabla \dd(\f x , t)\o \dd(\f x ,t)) \nu^o_{(\f x, t)} ( \de \f S)
\\\geq \left ( \nabla \f d (\f x ,t) \o \f d(\f x,t) - \nabla \dd(\f x ,t) \o \dd(\f x ,t) \right ) \dreidots \f \Theta \dreidots\left ( \nabla \f d (\f x ,t) \o \f d(\f x,t) - \nabla \dd(\f x ,t) \o \dd(\f x ,t) \right )\,.
\end{multline*}

\end{proof}
\begin{remark}[Variance of the measure-valued solution ]
Similar to the expectation, which is the first moment of the probability measure, the centered second moment or variance can be calculated. 
The variance of a measure $\nu^o\in \mathcal{P}(\Omega;\R^{d\times d}) $ with expectation $\nabla \f d $, measured in the semi-norm $\sqrt{ \cdot : \f \Lambda : \cdot }$ on $\R^{d\times d }$, is given by 
\begin{align*}
\sigma(\nu^o) = \int_{\R^{d\times d} } ( \f S- \nabla \f d  ):\f \Lambda :  ( \f S- \nabla \f d  ) \nu^0_{(\f x ,t)}(\de \f S) = \int_{\R^{d\times d} }  \f S:\f \Lambda :   \f S \nu^0_{(\f x ,t)} (\de \f S) - \int_{\R^{d\times d} } \nabla \f d  :\f \Lambda :   \nabla \f d   \nu^0_{(\f x ,t)}(\de \f S) \,.
\end{align*}
The second equality can be verified by~\eqref{identify}. 
The inequality of Proposition~\eqref{prop:relenab} can be seen as a gab induced by the lack of regularity. If  this inequality becomes an equality, this would imply that the defect measure $m_t$ vanishes and that the variation~$\sigma(\nu^0)$ vanishes. 
This would further imply that  the oscilation measure is just a point measure, i.e., $\nu^o_{(\f x ,t)}  = \delta _{\nabla \f d(\f x ,t)}$. 
\end{remark}
\begin{lemma}\label{lem:timederi}
Let $\f d $ be a measure-valued solution according to Definition~\ref{def:meas} and $\dd$ an appropriate test function admitting the regularity~\eqref{regtest} as well as $\f a \in L^1(0,T, \f L^3)$. 
For every $\delta>0$ there exists a $C_\delta>0$ such that
\begin{multline*}
\int_0^t \left (  {\dd(s)}\times  \t \dd(s)  -  {\f d(s) } \times \t \f d(s)  ,\left ( {\f d(s) } - {\dd(s)} \right )\times \f a(s) \right ) \de s\\ \leq  \int_0^t\left [\mathcal{K}(s)\mathcal{E}(s)  + \delta  \mathcal{W}(s)\right ]\de s+ \int_0^t\left (\mathcal{A}_2(s) , ({ \f d(s)  -\dd(s)  })\times  \f a(s) \right ) \,. 
\end{multline*}
Here $\mathcal{K}$ is given by $\mathcal{K}= C_\delta \left ( \| \vv\| _{\f L^\infty}^2 + \| \vv\|_{\f W^{1,3}}^2+ \| \f a \|_{L^3}^2\right )(\| \f d \|_{L^\infty(\f L^\infty)}\| \dd \|_{L^\infty(\f L^\infty)}^2  \| \nabla \dd \|_{L^\infty(\f L^{3})}^2+1)$.
By $\mathcal{A}_2$, we denote the second component of~\eqref{A}, given by
\begin{align}
\mathcal{A}_2(s):=  {\dd(s)}\times \left ( \t \dd(s) + (\vv(s)\cdot \nabla) \dd (s)- (\nabla \vv(s))_{\skw} \dd(s) + \lambda (\nabla \vv(s))_{\sym} \dd(s) + \tq (s)\right )\,.\label{A2}
\end{align}
\end{lemma}
\begin{proof}
First, we insert equation~\eqref{eq:mdir}  for the measure-valued solution and add and subtract  simultaneously \eqref{dir} for the test function. This gives
\begin{subequations}\label{gleichung}
\begin{align}
\int_0^t &\left (  \rot{\f d      -\dd      }^T \left ( \rot{\f d       } \t \f d       - \rot{\dd      } \t \dd       \right ) , \f a       \right ) \de s \notag\\
={}& \int_0^t \left ( \left ( \rot{\f d       } - \rot{\dd      } \right )^T \left ( \rot{\dd      } ( \vv       \cdot \nabla ) \dd      - \rot{\f d       } ( \f v        \cdot \nabla ) \f d    +\rot{\f d       } \sk v  \f d -\rot{\dd      } \skv \dd      \right ) , \f a       \right ) \de s
\label{gleichung2}\\
&+  \int_0^t \left ( \left ( \rot{\f d}     -\rot{\dd      } \right )^T \left ( \lambda\rot{\dd      } \syv \dd      - \lambda\rot{\f d       } \sy v\f d    +\rot{\dd      } \tq      - \rot{\f d       } \f q          \right ) , \f a       \right ) \de s
\label{gleichung4}
\\
&+ \int_0^t\left ( \rot{\dd} \left ( \t \dd  + (\vv \cdot \nabla) \dd  - (\nabla \vv )_{\skw} \dd  + \lambda (\nabla \vv )_{\sym} \dd  + \tq  \right ) , \rot{\dd  -\f d } \f a  \right )\,.\notag
\end{align}
\end{subequations}
The dependence on $s$ is not written out to remain the lucidity.
The terms in the lines~\eqref{gleichung2} and~\eqref{gleichung4} can be estimated in the same way as in~\cite[Lemma~4.2]{weakstrong}. This gives the assertion. 

\end{proof}
\subsection{Integration-by-parts formulae\label{sec:int}}
\begin{proposition}[Integration-by-parts fomula]\label{prop:int}
Let $(\f v, \f d)$ be a measure-valued solution according to Definition~\ref{def:meas} and $( \vv , \dd )$ an appropriate test function admitting the regularity~\eqref{regtest}.
 Then the  integration-by-parts formulae 
 \begin{subequations}\label{intd}
\begin{align}
( \f v ( t) , \vv(t)) - ( \f v (0) , \vv(0)) ={}& \int_0^t( \f v ( s) , \t \vv (s) ) + ( \t \f v (s) , \vv(s)) \de s \, ,\label{intv}\\
\frac{1}{2}\| \f d(t) - \dd(t) \|_{\Le}^2- \frac{1}{2}\| \f d(0) -\dd(0)\|_{\Le}^2  ={}& \int_0^t \left (  \f d(s) \times \t \f d(s) - \dd(s) \times \t \dd (s), \dd (s)\times \f d (s)  \right )\de s 
 \, , \label{intd1}
\end{align}
\begin{align}
\begin{split}
\left ( \nabla \f d(t) ;\, \f \Lambda : \nabla \dd(t)\right ) &- \left (  \nabla \f d(0) ;\, \f \Lambda : \nabla\dd(0)\right )\\ 
\geq{}&  \inttet{\left [ \ll{\nu_t,  \f S   :\f \Lambda :\left (\f \Upsilon : \left (\rot{\dd}\t \dd \o \f S \right ) \right )}+ \left ( \nabla \f d ; \f \Lambda : \rot{\f d }^T \nabla ( \rot{\dd} \t \dd )   \right )-\left ( \rot{\dd}  \Lap \dd , \rot{\f d } \t \f d  \right )\right ]}\\
&
- \inttet{\mathcal{K}\mathcal{E}}+ \inttet{\left (( \rot{\dd}-\rot{\f d } )^T ( \rot{\dd}\t \dd - \rot{\f d }\t \f d ) , - \Lap \dd \right )}\,,
\end{split}\label{intd2}
\end{align}
 and
\begin{align}
\begin{split}
\big ( \nabla \f d(t) \o \f d (t) \dreidotkom \f \Theta&  \dreidots \nabla \dd(t) \o \dd (t) \big )
-
\left ( \nabla \f d(0) \o \f d (0) \dreidotkom \f \Theta \dreidots \nabla \dd(0) \o \dd (0) \right )\\
&-
 \left ((\nabla \f d (t)-\nabla \dd(t)  )\o (\f d(t) -  \dd(t) )  \dreidots \f \Theta \dreidots \nabla \dd(t)  \o \dd(t)\right  )
  \\&+\left ((\nabla \f d (0)-\nabla \dd(0)  )\o (\f d(0) -  \dd(0)  ) \dreidots \f \Theta \dreidots \nabla \dd(0)  \o \dd(0) \right )
\\
\geq {}& \inttet{\left ( \rot{\f d} \t \f d , \rot{\dd} \left (- \di \left ( \dd \cdot \f \Theta \dreidots \nabla \dd \o \dd \right )+ \nabla \dd : \f \Theta \dreidots \nabla \dd \o \dd \right ) \right )}
\\
&+ \inttet{ \ll{\nu_\tau , \f S \o \f h \dreidots \f \Theta \dreidots \left (\f \Upsilon : ( \rot \dd^T  \t \dd \o \f S )\right ) \o \f h }}
\\&+ \inttet{ \left ( \nabla \f d \o \f d \dreidotkom \f \Theta \dreidots \rot{\f d}^T \nabla \left (\rot \dd \t \dd \right ) \o \f d \right )}
\\
&+ \inttet{ \ll{\nu_\tau , \f S \o \f h \dreidots \f \Theta \dreidots \f S \o \rot{\f h}^T \rot \dd \t \dd }} - c \inttet{\mathcal{K}\mathcal{E}}
\\
&+\inttet{\left ( \rot{\f d} \t \f d - \rot \dd \t \dd , \left ( \rot{\f d}-\rot{\dd}\right ) \left (- \di \left ( \dd \cdot \f \Theta \dreidots \nabla \dd \o \dd \right )+ \nabla \dd : \f \Theta \dreidots \nabla \dd \o \dd \right ) \right )} \,.
\end{split}\label{intdd}
\end{align}
\end{subequations}
hold for a.e. $t\in(0,T)$. 
In the formulas~\eqref{intd2} and~\eqref{intdd}, the dependence on $s$ under the time integral is not written out to remain the lucidity.
\end{proposition}
\begin{proof}
The proof is similar to the one of~\cite[Proposition~5.1]{weakstrong}.
Formula~\eqref{intv} is already proven there.
To get formula~\eqref{intd1}, we observe for $\phi \in \C_c^\infty(0,T)$ that 
\begin{align}
\begin{split}
-\frac{1}{2}\intte{\phi'\left  \| \f d - \dd \right \|_{\Le}^2 } ={}& \intte{ \phi \left ( \t \f d - \t \dd , \f d - \dd \right )}=   \intte{ \phi \left ( \rot{\f d}^T \rot{\f d} \t \f d - \rot \dd ^T \rot \dd \t \dd , \f d - \dd \right )}\\
={}&  \intte{\phi \left (  \rot {\f d} \t \f d , \rot{\f d- \dd} ( \f d - \dd)\right )} +  \intte{\phi \left ( \rot { \f d} \t \f d - \rot\dd \t \dd , (\rot{ \f d} - \rot \dd ) \dd \right )}\,.
\end{split}\label{equch}
\end{align} 
The first equality is an integration-by-parts as well as the chain rule, the second equation is valid due to~\cite[Lemma~3.2]{weakstrong}, and the last equation follows from a rearrangement. 
The first term on the right-hand side of the above equality chain vanishes, since $ \rot{\f a}\f a=0$ for any vector $\f a \in \R^3$.
From the definition of the weak-derivative, we infer that the left-hand side of the equality chain~\eqref{equch} is the weak derivative of the right-hand side. Thus, it is absolutely continuous and formula~\eqref{intd1} holds.  
For the proofs of the integration-by-parts formulae~\eqref{intd2} and~\eqref{intdd}, we refer again to~\cite[Proposition~5.1]{weakstrong}. Note the last terms on the right-hand side of inequality~\eqref{intd2} and~\eqref{intdd} depending on the difference of the time derivatives of $\f d $ and $\dd$ are in comparison to~\cite[Proposition~5.1]{weakstrong} not estimated by~\cite[Lemma~4.2]{weakstrong}, but kept on the right hand side. 
\end{proof}


\begin{corollary}\label{cor:variational}
Let $(\f v, \f d)$ be a measure-valued solution and $( \vv , \dd )$ be a test function admitting the regularity~\eqref{regtest}. Then it holds 
\begin{align}
\begin{split}
- &\left ( \nabla \f d (t) ; \f \Lambda : \nabla \dd(t)\right ) - \left (  \nabla \f d(t)\o \f d(t) \dreidotkom \f \Theta \dreidots \nabla \dd(t) \o \dd(t)  \right ) 
\\
\leq {}& -  \int_0^t   \left [ ( \f d \times \t \f d  ,\dd \times  \tq) + ( \f d \times \f q ,\dd \times  \t \dd)  \right ]    \de s \\&- 
\left  ( \nabla \f d(0); \f \Lambda : \nabla \dd(0)\right ) - \left ( \nabla \f d(0) \o \f d(0) \dreidotkom \f \Theta \dreidots \nabla \dd(0) \o \dd(0)\right )\\
 & + \frac{1}{2} \mathcal{E}^M(t)  + \int_0^t \mathcal{K}(s) \mathcal{E}^M(s) \de s  +\inttet{\left (  \rot{\dd}\t \dd -  \rot{\f d }\t \f d   ,  \rot{\f d-\dd } \tq+ \frac{|\f \Theta|^2}{k}\| \nabla \dd \o \dd \|_{L^\infty( \f L^\infty)}^2 \f d \times \dd   \right )}\\
&
+\left ((\nabla \f d(0)-\nabla \dd (0))\o (\f d(0)-  \dd(0) ) \dreidots \f \Theta \dreidots \nabla \dd (0)\o \dd(0)\right )
+ \frac{|\f \Theta|^2}{2k}\| \nabla \dd \o \dd \|_{L^\infty( \f L^\infty)}^2 \left \|   \f d(0)- \dd(0)    \right \|_{\Le}^2 
\end{split}\label{auscor}
\end{align}
 for a.e.~$t\in (0,T)$.

\end{corollary}

\begin{proof}
We add the two integration by parts formulae~\eqref{intd2} and~\eqref{intdd}. Recall the definitions of $\f q$ and $\tq$ in~\eqref{eq:q} and~\eqref{qdef}, respectively. 
For the term in the second line of~\eqref{intdd}, we observe similar to~\cite[Corollary~5.1]{weakstrong} with~\eqref{identify} and Young's inequality that 
\begin{align}
\begin{split}
&\left ((\nabla \f d(t)-\nabla \dd (t))\o (\f d(t)-  \dd(t)) \dreidotkom  \f \Theta \dreidots \nabla \dd (t)\o \dd(t) \right )\\
&\qquad= \int_{\Omega} \int_{\R^{3\times 3}}  \left ( ( \f S - \nabla \dd (\f x, t) ) \o ( \f d (\f x, t) - \dd (t) ) \dreidots \f \Theta \dreidots   \nabla \dd (\f x, t) \o \dd (\f x ,t) \right )  \nu^o _{( \f x ,t)}(\de \f S) \de \f x  \\
&\qquad\leq  \frac{k}{2}\int_{\Omega} \int_{\R^{3\times 3}}  \left | \f S - \nabla \dd (t) \right |^2 \nu^o _{( \f x ,t)}(\de \f S) \de \f x + \frac{|\f \Theta|^2}{2k} \left \|\nabla \dd\o \dd  \right \|^2_{L^\infty( \f L^\infty)} \int_{\Omega}| \f d (t) - \dd (t)|^2  \de \f x  \\
&\qquad \leq \frac{k}{2}\ll{\nu_t,  | \f S - \nabla \dd (t) |^2  }   +  \frac{|\f \Theta|^2}{2k} \left \|\nabla \dd\o \dd  \right \|^2_{L^\infty( \f L^\infty)} \left\| \f d (t) - \dd (t) \right \|_{\Le}^2   \,.
\end{split}\label{Youngint}
\end{align}
Remark that the defect measure $m_t$ is a positive measure, $m_t \in \M^+(\ov \Omega)$.
The constant $k$ is chosen small enough, such that  (compare~\eqref{kill})
\begin{align*}
k\ll{  \nu _{t}, | \f S - \nabla \dd (t) |^2   }  \leq \frac{1}{2}\ll{\nu _{t},  ( \f S - \nabla \dd (t)) :\f \Lambda : ( \f S - \nabla \dd (t))  }  \leq  \mathcal{E}^M(t)\,.
\end{align*} 
Applying formula~\eqref{intd1} gives now the assertion. 
\end{proof}

\subsection{Relative energy inequality\label{sec:main}}
\begin{proof}[Proof of Theorem~\ref{thm:main}]
Considering the relative measure-valued energy, we observe
\begin{align}
\mathcal{E}^M(t) ={}& \ll{\mu_t , 1 }+\ll{\nu_t , F } + \frac{1}{2}\| \f v(t)\|_{\Le}^2  + \F( \dd(t)) + \frac{1}{2}\| \vv(t)\|_{\Le}^2\notag \\
&- ( \f v(t) , \vv (t))  - \left ( \nabla \f d (t) ;\, \f \Lambda : \nabla \dd(t)\right ) - \left (  \nabla \f d(t)\o \f d(t) \dreidotkom \f \Theta \dreidots \nabla \dd(t) \o \dd(t)  \right )  \,.\label{relEnM}
\end{align}
Inserting the integration-by-parts formula~\eqref{intv} and Corollary~\ref{cor:variational}, we find
\begin{align*}
\mathcal{E}^M(t) 
\leq {}&\ll{\mu_t , 1 }+\ll{\nu_t , F } + \frac{1}{2}\| \f v(t)\|_{\Le}^2  + \F( \dd(t)) + \frac{1}{2}\| \vv(t)\|_{\Le}^2 \\
{}& -  \int_0^t\left [( \f v  , \t \vv  ) + ( \t \f v  , \vv) \right ]\de s -  \int_0^t   \left [ ( \f d \times \t \f d  ,\dd \times  \tq) + ( \f d \times \f q ,\dd \times  \t \dd)  \right ]    \de s \\&- ( \f v(0), \vv(0))- 
\left  ( \nabla \f d(0); \f \Lambda : \nabla \dd(0)\right ) - \left ( \nabla \f d(0) \o \f d(0) \dreidotkom \f \Theta \dreidots \nabla \dd(0) \o \dd(0)\right )\\
 & + \frac{1}{2} \mathcal{E}^M(t)  +  \int_0^t \mathcal{K}(s)\mathcal{E}^M(s) \de s  +\inttet{\left (  \rot{\dd}\t \dd - \rot{\f d }\t \f d  ,  \rot{\f d -\dd} \tq+ \frac{2|\f \Theta|^2}{k}\| \nabla \dd \o \dd \|_{L^\infty( \f L^\infty)}^2 \f d \times \dd   \right )}\\
&
+\left ((\nabla \f d(0)-\nabla \dd (0))\o (\f d(0)-  \dd(0))\dreidotkom \f \Theta \dreidots \nabla \dd (0)\o \dd(0)\right ) 
+ \frac{|\f \Theta|^2}{2k}\| \nabla \dd \o \dd \|_{L^\infty( \f L^\infty)}^2 \left \|   \f d(0)- \dd(0)    \right \|_{\Le}^2 \,.
\end{align*}
Inserting the energy inequality~\eqref{energyin} for the measure-valued solution, the shifted energy equality (see~Proposition~\ref{prop:shift}), and adding the relative dissipation~\eqref{relW} on the left-hand side,  yields
\begin{align}
\begin{split}
\mathcal{E}^M(t) + \int_0^t\mathcal{W}(s) \de s \leq{}& \F( \f d_0) + \frac{1}{2}\| \f v_0\|_{\Le}^2 + \F( \dd_0) + \frac{1}{2}\| \vv_0\|_{\Le}^2 \\
&- 2 (\mu_1 +\lambda^2 )\int_0^t ( \f d \cdot \sy v \f d , \dd\cdot \syv  \dd )   \de s   \\
& - 2( \mu_5 + \mu_6 - \lambda ^2) \int _0^t ( \sy v \f d , \syv \dd ) \de s \\
& - 2 \mu_4 \int_0^t ( \sy v ; \syv ) \de s - 2  \int_0^t ( \f d\times \f q , \dd \times \tq ) \de s  + \int_0^t \langle \f g , \f v + \vv \rangle \de s \\
& -  \int_0^t\left [( \f v  , \t \vv  ) + ( \t \f v  , \vv) \right ]\de s -  \int_0^t   \left [ ( \f d \times \t \f d  ,\dd \times  \tq) + ( \f d \times \f q ,\dd \times  \t \dd)  \right ]    \de s 
\\&- ( \f v(0), \vv(0))- 
\left  ( \nabla \f d(0); \f \Lambda : \nabla \dd(0)\right ) - \left ( \nabla \f d(0) \o \f d(0) \dreidotkom \f \Theta \dreidots \nabla \dd(0) \o \dd(0)\right )\\
&
+\left ((\nabla \f d(0)-\nabla \dd (0))\o (\f d(0)-  \dd(0))\dreidotkom \f \Theta \dreidots \nabla \dd (0)\o \dd(0)\right ) + \frac{1}{2} \mathcal{E}^M(t)  +  \int_0^t\mathcal{K}(s) \mathcal{E}^M(s) \de s 
\\
 &  +\inttet{\left ( \rot{\dd}\t \dd  - \rot{\f d }\t \f d ,  \rot{\f d-\dd } \tq+ \frac{|\f \Theta|^2}{k}\| \nabla \dd \o \dd \|_{L^\infty( \f L^\infty)}^2 \f d \times \dd   \right )}\\
&+ \frac{|\f \Theta|^2}{2k}\| \nabla \dd \o \dd \|_{L^\infty( \f L^\infty)}^2 \left \|   \f d(0)- \dd(0)    \right \|_{\Le}^2  + \inttet{ \left (\mathcal{A}, \begin{pmatrix}
 \vv \\ \dd \times \tq
 \end{pmatrix}\right )}
 \, .
\end{split}
\label{vormeins}
\end{align}

Regarding the terms incorporating the initial values, we observe
\begin{align*}
\F( \f d_0) + \frac{1}{2}\| \f v_0\|_{\Le}^2 + \F( \dd_0) + \frac{1}{2}\| \vv_0\|_{\Le}^2- ( \f v(0), \vv(0)) - 
\left  ( \nabla \f d(0); \,\f \Lambda : \nabla \dd(0)\right )&\\ - \left ( \nabla \f d(0) \o \f d(0) \dreidotkom \f \Theta \dreidots \nabla \dd(0) \o \dd(0)\right ) &= \mathcal{E}(0)\, .
\end{align*}
In the next step, we use that $(\f v, \f d)$ is a measure-valued solution.
We consider equation~\eqref{eq:velo} tested with $\vv$ and the equation~\eqref{eq:mdir} tested with $\dd \times \tq$ added up:
\begin{align*}
-{}&\inttett{( \t \f v , \vv) + ( \f d \times \t \f d , \dd \times \tq) + \mu_4 \left ( \sy v ; \syv \right ) + \left ( \f d \times \f q , \dd \times \tq\right )}\\
={}& \inttett{\left ( ( \f v \cdot \nabla ) \f v , \vv  \right )+ \mu_1 \left ( \f d \cdot \sy v , \f d \cdot\syv \f d\right ) + ( \mu_5+\mu_6) \left ( \sy v \f d , \syv \f d \right ) }\\
& + \inttett{\lambda \left ( \syv \f d , \f e\right ) + \lambda \left ( \sy v \f d , \skv \f d  \right ) + \left (\f e , \skv \f d \right )}\\
& -\inttett{\ll{\nu_t , \f S^T F_{\f S}( \f h ,\f S) : \nabla \vv }+ 2 \ll{\mu_t, \f \Gamma \dreidots (\f \Gamma\cdot \nabla \vv) }}\\
&+ \inttett{\left ( \rot{\f d }( \f v \cdot \nabla ) \f d , \rot \dd \tq \right ) - \left ( \rot{\f d} \sk v \f d , \rot{\dd} \tq \right )+ \lambda  \left ( \rot{\f d} \sy v \f d , \rot\dd \tq\right )}\,.
\end{align*}
Due to \cite[Corollary~3.1]{weakstrong}, it holds that $ \f e = -\rot{\f d}^T \rot{\f d} (\lambda \sy v \f d + \f q )$. Replacing $\f e$ in the above equation and using $\rot{\f d}^T\rot{\f d}= I-\f d \o \f d$ yields
\begin{align}
\begin{split}
-{}&\inttett{( \t \f v , \vv) + ( \f d \times \t \f d , \dd \times \tq) + \mu_4 \left ( \sy v ; \syv \right ) + \left ( \f d \times \f q , \dd \times \tq\right )}\\
={}& \inttett{\left ( ( \f v \cdot \nabla ) \f v , \vv  \right )+ (\mu_1+\lambda^2) \left ( \f d \cdot \sy v , \f d \cdot\syv \f d\right ) + ( \mu_5+\mu_6-\lambda^2 ) \left ( \sy v \f d , \syv \f d \right ) }\\
& - \inttett{\lambda \left (\f d \times  \syv \f d ,\f d \times  \f q\right )
+ \left (\f d \times \f q ,\f d \times  \skv \f d \right )+\ll{\nu_t , \f S^T F_{\f S}( \f h ,\f S) : \nabla \vv }+ 2 \ll{\mu_t, \f \Gamma \dreidots( \f \Gamma\cdot \nabla \vv) }}\\
&+ \inttett{\left ( \rot{\f d }( \f v \cdot \nabla ) \f d , \rot \dd \tq \right ) - \left ( \rot{\f d} \sk v \f d , \rot{\dd} \tq \right )+ \lambda  \left ( \rot{\f d} \sy v \f d , \rot\dd \tq\right )}\,.
\end{split}\label{measeq}
\end{align}
Remark that $\f d\cdot \skv \f d=0$. 
The main difference to the weak-strong uniqueness proof in~\cite{weakstrong} is, that we now add a zero to the right hand side by simultaneously adding and subtracting equation~\eqref{nav} and equation~\eqref{dir} evaluated at $(\vv,\dd)$ and tested with  $\f v$ and  $\f d\times \f q$, respectively. 
Replacing again $\tilde{\f e}$ by $ -(I-\dd \o \dd)(\lambda \syv \dd + \tq)$ leads to 
\begin{align}
\begin{split}
-{}& \inttett{\left ( \t \vv ,\f v \right ) + \left ( \dd \times \t \dd , \f d \times \f q\right ) + \mu_4 \left ( \syv, \sy v \right ) + \left ( \dd \times \tq , \f d \times \f q \right ) }\\
={}& \inttett{\left ( ( \vv \cdot \nabla ) \vv , \f v  \right )+ (\mu_1+\lambda^2) \left (\dd \cdot \syv \dd , \dd \cdot \sy v \dd \right ) + ( \mu_5+ \mu_6-\lambda^2) \left ( \syv \dd , \sy v \dd \right ) }\\
&-\inttett{\lambda \left (\dd\times \tq ,\dd \times  \sy v \dd \right )
+ \left (\dd\times\tq , \dd\times\sk v \dd\right )+\left (\rot \dd\left ( ( \vv \cdot \nabla) \dd -\skv \dd +\lambda \syv \dd\right ), \rot \dd  \f q \right )}\\
&- \inttett{  \left (\nabla \dd ^T F_{\f S}( \dd , \nabla \dd) ; \nabla \f v \right )+\left (\mathcal{A}(s), \begin{pmatrix}
\f v \\ \f d \times \f q 
\end{pmatrix} \right )}
\,.
\end{split}\label{strongeq}
\end{align}
Remark that $\dd\cdot \sk v \dd =0$.
Inserting the equations~\eqref{measeq} and~\eqref{strongeq} in~\eqref{vormeins} and using Lemma~\eqref{lem:timederi} yields 
\begin{align*}
\begin{split}
\frac{1}{2}\mathcal{E}^M(t)& + \int_0^t\mathcal{W}(s) \de s 
\\
\leq{}& \mathcal{E}(0)  + \delta\int_0^t\mathcal{W}(s) \de s +  \int_0^t\mathcal{K}(s) \mathcal{E}^M(s)\de s + \int_0^t \left [( ( \vv \cdot \nabla ) \vv , \f v ) + ( ( \f v \cdot\nabla ) \f v , \vv ) \right ]\de s \\
& + (\mu_1 +\lambda^2)\int_0^t  \left [  \left ( \dd \cdot \syv \dd , \sy v : \left ( \dd \o \dd - \f d \o \f d   \right )   \right ) +  \left ( \f d  \cdot \sy v \f d  , \syv : \left ( \f d  \o \f d  - \dd \o \dd   \right )   \right ) \right ] \de s \\
& + ( \mu_5 + \mu_6  -\lambda^2) \int_0^t  \left [\left ( \syv \dd , \sy v ( \dd - \f d )\right  )  + \left (    \sy v \f d , \syv ( \f d - \dd ) \right ) \right ]   \de t\\
& +\inttet{\left [ \left (\dd \times \skv \dd- \f d \times \skv \f  d , \f d \times \f q \right )+ \left (\f d \times \sk v \f d- \dd \times \sk v \dd  , \dd \times \tq \right )  \right ]}\\
& -  \lambda \int_0^t \left [\left ( \dd\times( \sy v - \syv)\dd , \dd \times \tq \right )+ \left ( \f d \times (\syv - \sy v)\f d  , \f d \times \f q \right )\right ]  \de s
 \\
&+ \int_0^t \left [(\dd \times ( \vv \cdot \nabla ) \dd ,\f d \times \f q) + (\f d \times( \f v \cdot \nabla )\f d ,\dd \times  \tilde{\f q})\right ] \de s \\
&-\int_0^t\left [ ( \nabla \dd^T F_{\f S}(\dd,\nabla \dd) ; \nabla \f v )  + \ll{ \nu_s ,  \f S ^T F_{\f S}(\f h , \f S) : \nabla \vv    }\right ]\de s  -2\int_0^t \ll{\mu_s, \f \Gamma \dreidots   (\f \Gamma \cdot \nabla \vv) } \de s 
\\
&
+\left ((\nabla \f d(0)-\nabla \dd (0))\o (\f d(0)-  \dd(0))\dreidotkom \f \Theta \dreidots \nabla \dd (0)\o \dd(0)\right ) + \frac{|\f \Theta|^2}{k}\| \nabla \dd \o \dd \|_{L^\infty( \f L^\infty)}^2 \left \|   \f d(0)- \dd(0)    \right \|_{\Le}^2 
\\
& + \inttet{ \left (\mathcal{A}(s), \begin{pmatrix}
 \vv- \f v  \\ \dd \times \tq - \f d \times \f q + \rot{\f d- \dd  } \tq+ \frac{2|\f \Theta|^2}{k}\| \nabla \dd \o \dd \|_{L^\infty( \f L^\infty)}^2 \f d \times \dd 
 \end{pmatrix}\right )}
\end{split}
\end{align*}
The terms on the right-hand side independant of the initial values and $\mathcal{A}$  can be estimated by integrals over $\mathcal{E}^M$ and $\mathcal{W}$. This is done explicitly in~\cite[Proposition~6.1]{weakstrong}. 
Inserting these estimates, yields
\begin{align}\label{relEnergyinequ}
\begin{split}
\frac{1}{2}\mathcal{E}^M(t)& + \int_0^t\mathcal{W}(s) \de s 
\\
\leq{}& \mathcal{E}(0) +\left ((\nabla \f d(0)-\nabla \dd (0))\o (\f d(0)-  \dd(0))\dreidotkom \f \Theta \dreidots \nabla \dd (0)\o \dd(0)\right ) + \frac{|\f \Theta|^2}{k}\| \nabla \dd \o \dd \|_{L^\infty( \f L^\infty)}^2 \left \|   \f d(0)- \dd(0)    \right \|_{\Le}^2  
\\&
 + \delta\int_0^t \mathcal{W}(s)\de s+\int_0^t\mathcal{K}(s) \mathcal{E}^M(s)\de s   + \inttet{ \left (\mathcal{A}(s), \begin{pmatrix}
 \vv- \f v  \\  \f d \times (\tq -  \f q)- \frac{2|\f \Theta|^2}{k}\| \nabla \dd \o \dd \|_{L^\infty( \f L^\infty)}^2 \f d \times \dd 
 \end{pmatrix}\right )}
\end{split}
\end{align}
We choose $\delta=1/2$ and absorb the relative dissipation $\mathcal{W}$ on the left-hand side.
A version of Gronwall's inequality implies~\eqref{relenin} for the relative measure-valued energy $\mathcal{E}^M$ instead of $\mathcal{E}$.  Proposition~\ref{prop:relenab} implies the assertion.

\end{proof}
\begin{remark}
With the preceding proof, we especially showed that the regularized system in~\cite{masswertig} converges for vanishing regularization to a dissipative solution.
The numerical approximation of a solution via this regularization technique seems tedious.
One needs a higher order scheme to discretize the 4$^\text{th}$ order differential operator of the regularization. Additionally, one has to consider three different limits: The discretization limit as well as the vanishing regularization and penalization limit. The interchange or simultaneous convergence of this limits is not clear. 
Therefore, we propose in a future article~\cite{approx}  a semi-discrete approximation that fulfills the norm-restriction in every step and converges to a dissipative solution (see Definition~\ref{def:diss} and ~\cite{approx}).
\end{remark}
\begin{remark}
In view of the existence proof of measure-valued solutions~\cite{masswertig}, the existence of dissipative solutions~(see Definition~\ref{def:diss}) can also be proven via the regularization technique used in~\cite{masswertig}. Proving the relative energy inequality~\eqref{relenin} for the regularized system, one obtains dissipative solutions for vanishing regularization.  

\end{remark}

\section{Relative energy inequality for different boundary values and with electromagnetic field effects  \label{sec:5}}
In this section, we want to argue that the relative energy inequality even holds for solutions to different boundary values. 
For the velocity field $\f v$, we always assume homogeneous Dirichlet boundary conditions, i.e., $\tr(\f v) =0$. For the director field constant boundary values in time of a certain regularity and of unit length are assumed for the existence theory, i.e., $\tr (\f d) = \f d_1 $ with $|\f d_1|=1 $ and $\f d _1\in \Hrand{7}$. 
\subsection{Estimates for different boundary values}
In the proof of the relative-energy inequality in~\cite{weakstrong}, the only argument, which relies on the fact that the measure-valued solution~$\f d $ and the strong solution or test function~$\dd$ fulfill the same boundary conditions is the assertion of Lemma~\cite[Lemma~4.1]{weakstrong}.
There, we applied~\cite[Proposition~5.1]{masswertig} to the difference $\f d - \dd$, which fulfills homogeneous Dirichlet boundary conditions, if $\f d $ and $\dd$ fulfill the same boundary conditions. 
This can be generalized to non-equal boundary conditions, which is the assertion of the following Proposition that essentially states that an adapted version of Lemma~4.1 in~\cite{weakstrong} remains true if $\f d$ and $\dd $ do not fulfill the same boundary condition.
\begin{lemma}\label{Sobolev}
Let $\f d$ be a measure-valued solution (see~Definition~\ref{def:meas}), $( \nu,m,\nu^\infty) $ the associated generalized Young measure, and $\dd$ a function fulfilling~\eqref{regtest} as well as $|\tr(\dd)|=1$. Let the associated relative energy be given as above (see\eqref{relEnM}). Then there exists a constant $c>0$ such that
\begin{subequations}
\begin{align}
\int_0^t \| \f d(s)- \dd (s)  \|_{\f L^6}^2 \de s + \int_0^t \ll{\nu_s, | \f S - \nabla \dd(s) |^2}\de s \leq  c {}&\int_0^t \left [\mathcal{E}^M(s)+ \| \f d(s)-\dd(s)\|_{\Le}^2 \right ]\de s \,,\label{absch1}\\
\int_0^t \ll{\nu_s , \left | \f \Theta \dreidots \left ( \f S \o \f d - \nabla \dd (s)\o \dd(s) \right ) \right |^2 }\de s \leq c {}&\int_0^t \mathcal{E}^M(s)\de s \,,\label{absch2}\\
\int_0^t \ll{\nu_s , | \f S - \nabla \dd (s)|^2 |\f h - \dd(s) | ^2 } \de s \leq c(1 +\|\nabla \dd \|_{L^\infty(\f L^3)} ^2 + \| \dd \|_{L^\infty(\f L^\infty)}^2){}& \int_0^t \left [\mathcal{E}^M(s)+ \| \f d(s)-\dd(s)\|_{\Le}^2 \right ]\de s\,,\label{absch3}
\\
\int_0^t \| \f d (s) - \dd(s)\|_{\f L^{12}}^4 \de s \leq c(1 +\|\nabla \dd \|_{L^\infty(\f L^3)} ^2 + \| \dd \|_{L^\infty(\f L^\infty)}^2){}& \int_0^t \left [\mathcal{E}^M(s)+ \| \f d(s)-\dd(s)\|_{\Le}^2 \right ] \de s\,.\label{absch4}
\end{align}
\end{subequations}
\end{lemma}
\begin{proof}
The proof is similar to the proof of~\cite[Lemma~4.1]{weakstrong}. We only focus on the points that change in comparison to the previous result.
First, we observe that for a strongly elliptic tensor the associated Norm is coercive on $\Hb$, i.e., there exists a constant $c> 0$ such that  the estimate
\begin{align*}
\frac{1}{c}\| \nabla \f d \|_{\Le } ^2 \leq   \left ( \nabla \f d : \f \Lambda : \nabla \f d \right )  +  \| \f d \|_{\Le}^2 \,
\end{align*}
holds for all $\f d \in \He$ (see McLean~\cite[Theorem~4.6]{mclean}). 
For the estimate~\eqref{absch1}, we conclude in the same way as in the proof of~\cite[Lemma~4.1]{weakstrong}. Approximating the generalized Young measure $\nu_t$ by a sequence $\{\fn d\}_{n\in\N}\subset \He $ yields  
\begin{align*}
\| \nabla \fn d - \nabla \dd \|_{\Le}^2 \leq c \left ( \nabla \fn d -\nabla \dd  : \f \Lambda :( \nabla \fn d-\nabla \dd)  \right ) +c  \| \fn d - \dd\|_{\Le}^2   \,. 
\end{align*} 
Going to the limit $n\ra \infty$ in the above inequality using the convergence in the sense of generalized gradient Young measures and the strong convergence of $\fn d $ in $\Le$ gives 
\begin{align}\label{erster}
 \int_0^t \ll{\nu_s, | \f S - \nabla \dd(s) |^2}\de s \leq  c {}&\int_0^t \left [\mathcal{E}^M(s)+ \| \f d(s)-\dd(s)\|_{\Le}^2 \right ]\de s \,.
\end{align}
Jensens's inequality implies in the same way as in the proof of~\cite[Lemma~4.1]{weakstrong} that
\begin{align}\label{zweiter}
\| \nabla \f d(s) - \nabla \dd(s) \|_{\Le}^2 \leq  \ll{\nu_s, | \f S - \nabla \dd(s) |^2} \,.
\end{align}
The embedding in three space dimensions $\He \hookrightarrow \f L^6$ implies 
\begin{align}\label{dritter}
\| \f d(s) - \dd(s) \|_{\f L^6} ^2 \leq c \left ( \| \nabla \f d(s) - \nabla \dd(s) \|_{\Le}^2 + \| \f d (s) - \dd(s ) \|_{\Le}^2 \right ) \,.
\end{align}
The three inequalities~\eqref{erster},~\eqref{zweiter}, and~\eqref{dritter} yield the conclusion~\eqref{absch1}.

Inequality~\eqref{absch2} remains valid even for different boundary values since it relies on a purely algebraic relation (see~\cite[Proposition~A.1]{weakstrong}).

To get inequality~\eqref{absch3}, we use the integration by parts formula derived in~\cite[Propostion~5.1]{masswertig}. 
For  functions $\f a \in \Hc$, it holds 
\begin{align*}
\int_{\Omega} | \f a|^2 | \nabla \f a|^2 \de \f x & = \int_{\Omega} ( ( \di \f a)^2 | \f a|^2 + |\f a |^2| \curl \f a |^2 + (\tr ( \nabla \f a ^2 ) - ( \di \f a ) ^2 ) | \f a|^2 ) \de \f x \\
& = \int_{\Omega} ( ( \di \f a)^2 | \f a|^2 + ( \f a \cdot \curl \f a )^2 + | \f a \times \curl \f a |^2) \de \f x \\ 
& \quad + \int_{\Omega} \di ( ( \nabla \f a \f a - ( \di \f a ) \f a ) | \f a |^2)\de \f x \\
& \quad  +\int_{\Omega} | \f a \times \curl \f a|^2 - | \nabla \f a \f a|^2 - | (\nabla \f a)^T \f a  |^2   \de \f x\\ 
& \quad + \int_{\Omega} ( \di \f a) \f a \cdot \nabla \f a  \f a+ ( \di \f a) \f a \cdot (\nabla \f a)^T  \f a  \de \f x
 \,.
\end{align*}
Young's  inequality and the divergence theorem imply
\begin{align*}
\int_{\Omega} | \f a|^2 | \nabla \f a|^2 \de \f x 
& \leq  \int_{\Omega} \left ( \frac{3}{2}( \di \f a)^2 | \f a|^2 + ( \f a \cdot \curl \f a )^2 + 2| \f a \times \curl \f a |^2\right ) \de \f x \\ 
& \quad + \int_{\partial\Omega} \f n\cdot  ( ( \nabla \f a \f a - ( \di \f a ) \f a ) | \f a |^2)\de \f S \,.
\end{align*}
For any function~$ \f a \in \Hc $, we observe 
\begin{align*}
\int_{\Omega} | \f a|^2 | \nabla \f a|^2 \de \f x 
 \leq {} & 4 \left ( \nabla \f a \o \f a \dreidotkom \f \Theta \dreidots \nabla \f a\o \f a \right )  
+  \langle \f \gamma_{\f n} (\nabla \f a ), \f \gamma_0( \f a | \f a|^2   ) \rangle _{\Hrand{-1},\Hrand{1}}\\
 &+ \langle \f \gamma_{\f n} ( \f a| \f a|^2 ), \f \gamma_0( \di \f a  ) \rangle _{\Hrand{1},\Hrand{-1}}  \\
 \leq{}& 4 \left ( \nabla \f a \o \f a \dreidotkom \f \Theta \dreidots \nabla \f a\o \f a \right )   + \| \nabla \f a \|_{\Hrand{-1}} \| \f a | \f a |^2 \|_{\Hrand{1}} + \| \di \f a \|_{\Hrand{-1}}\| \f a | \f a|^2 \|_{\Hrand{1}}\,.
\end{align*}
As long as $ | \f a |\leq c$ for a constant $c>0$, we observe
$ \| \f a| \f a |^2 \|_{\Hrand{1}}  \leq 3 c ^2  \| \f a\|_{\Hrand{1}} \,.
$
The continuous embeddings $\He \hookrightarrow \Hrand{1}$ and $\Le \hookrightarrow \Hrand{-1}$ (see~\cite[Theorem~3.37]{mclean}) guarantee that
\begin{align*}
\int_{\Omega} | \f a|^2 | \nabla \f a|^2 \de \f x 
 \leq {} & 4 \left ( \nabla \f a \o \f a \dreidotkom \f \Theta \dreidots \nabla \f a\o \f a \right )   + c \| \nabla \f a \|_{\Le}^2+ c  \| \f a  \|_{\He} ^2 \,.
\end{align*}
Approximating again the generalized Young measure $\nu_t$ by a sequence $\{\fn d (t)\}_{n \in \N} \subset \Hc$ (compare to~\cite[Definition~2]{weakstrong}), we find
\begin{align*}
\left \|\nabla \fn d -\nabla \dd | \fn d - \dd|\right \|_{\Le}^2 \leq {}& c \left ( (\nabla \fn d -\nabla \dd ) \o ( \fn d -\dd) \dreidotkom \f\Theta \dreidots  ( \nabla \fn d -\nabla \dd ) \o ( \fn d -\dd) \right ) \\&+ c\| \nabla \fn d-\nabla \dd\|_{\Le}^2 + c\| \fn d - \dd\|_{\Le}^2\,. 
\end{align*}
Note that the boundary values are regular enough such that the approximating sequence can be chosen in a way that its elements fulfill the same boundary conditions $\tr(\fn d ) = \f d_1$ in $\Hrand{1}$. 
Going to the limit in the approximation of the generalized Gradient Young measure implies
\begin{align*}
\ll{\nu_t , | \f S - \nabla \dd |^2 | \f h - \dd |^2 } \leq{}& c \ll{\nu_t , ( \f S - \nabla\dd ) \o ( \f h - \dd ) \dreidots \f \Theta \dreidots ( \f S - \nabla\dd ) \o (\f h - \dd )} \\
&+c \ll{\nu_t , | \f S -\nabla \dd|^2 }+ c \| \f d - \dd\|_{\Le}^2\,.
\end{align*}
The same rearrangements  as in the proof of~\cite[Lemma~4.1]{weakstrong} lead to~\eqref{absch3}. 

Finally, we observe with the embedding $\He \hookrightarrow \Le$  in three space dimensions that
\begin{align}\label{L12}
\Big \| | \f d - \dd | ^2 \Big \|_{\f L^6}^2 \leq  c \left (\left \| \nabla | \f d - \dd|^2\right \|_{\Le}^2 +\left \| |  \f d - \dd |^2 \right \|_{\Le}^2 \right ) \,.
\end{align}
Taking the derivative of the absolute value, we observe for the first term on the right-hand side that
\begin{align*}
\left \| \nabla | \f d - \dd|^2\right \|_{\Le}^2 \leq  \left \| 2 ( \nabla \f d - \nabla \dd )^T (\f d - \dd) \right  \|_{\Le}^2 \leq 4 \left \| \nabla \f d - \nabla \dd | \f d -\dd| \right \|_{\Le}^2  \,.
\end{align*}
The positivity of the defect measure and Jensen's inequality yields
\begin{align*}
  \left \| \nabla \f d - \nabla \dd | \f d -\dd| \right \|_{\Le}^2 \leq \ll{\nu_t , | \f S - \nabla \dd |^2 | \f h - \dd |^2 }  \,.
\end{align*}

To estimate  the second term on the right-hand side of~\eqref{L12}, we  adopt the Gagliardo-Nirenberg inequality 
\begin{align*}
\| \f d -\dd \|_{\f L^4} ^2 \leq  c \| \f d -\dd\|_{\He}^{4/3} \| \f d -\dd\|_{\f L^2}^{2/3}  \leq c \left ( \| \nabla \f d - \dd \| _{\Le}^2  + \| \f d -\dd\|_{\Le}^2 \right ) \leq c \left ( \ll{\nu_t , | \f S- \nabla \dd |^2 } + \| \f d - \dd \|_{\Le}^2\right )  \,.
\end{align*}
The second inequality is an application of Young's inequality and for the third one we employ~\eqref{zweiter}. 
\end{proof}
\begin{corollary}\label{cor:rand}
Let $\Omega\subset \R^{3}$ be of class $\C^{1,1}$. 
Then there exists a constant $c>0$ such that
\begin{align*}
\| \f d -\dd \|_{\Le } ^2 \leq c \|\nabla  \f d -\nabla  \dd \|_{\Le}^2 + c \| \tr(\f d -\dd) \|_{\Hrand{1}}^2 \,.
\end{align*}
for all $\f d$, $\dd \in \He$.
\end{corollary}
\begin{proof}
We employ a result on extension operators similar to the one in~\cite[Theorem~4.1]{masswertig}.
First, we note that the trace is continuous as a mapping between $ \He $ and $\Hrand{1}$. It is even surjective. 

Conversely,  there exists a linear continuous operator $\Sr: \f H^{\nicefrac{1}{2}}(\partial \Omega)  \ra \f H^{1}(\Omega)$, where $\Omega$ is of class $\C^{1,1}$.
This operator is the right-inverse of the trace operator, i.e.~for all $\f g\in \f H^{\nicefrac{1}{2}}(\partial \Omega)  $, it holds  $ \Sr\f g = \f g $  on $\partial\Omega$ in the sense of the trace operator.
Therewith, we can estimate with Poincar\'{e}'s inequality
\begin{align*}
\| \f d -\dd \|_{\Le} \leq{}& \| \f d - \Sr (\tr(\f d)) - ( \dd - \Sr(\tr(\dd)) \|_{\Le} + \| \Sr (\tr(\f d)) -\Sr (\tr (\dd)) \|_{\Le} \\  \leq{}& c \|\nabla ( \f d - \Sr (\tr(\f d))) - \nabla ( \dd - \Sr(\tr(\dd)) \|_{\Le}  + \| \Sr (\tr(\f d)) -\Sr (\tr (\dd)) \|_{\Le} \\  \leq{}& c \| \nabla \f d - \nabla \dd \|_{\Le} + c  \| \Sr (\tr(\f d)) - \Sr (\tr(\dd)) \|_{\He} \\
\leq {}& c \| \nabla \f d - \nabla \dd \|_{\Le} + c \| \tr(\f d) - \tr( \dd) \|_{\Hrand{1}} \,.
\end{align*}

\end{proof}
\begin{remark}
It is also possible to 
prove the dependence on the difference of the boundary values in the $\Hrand{1}$- norm instead of the difference of the function in the $\Le$-norm immediately 
in Lemma~\ref{Sobolev}. 
We exemplify the calculations for the inequality~\eqref{absch1}. Indeed, similar to the proof of~\cite[Proposition~5.1]{masswertig}, we observe that 
\begin{align*}
\|\nabla \f d -\nabla  \dd \|_{\Le}^2 \leq  \| \di ( \f d -\dd) \|_{L^2}^2 + \| \curl (\f d -\dd) \|_{\Le}^2 &+ \langle \f \gamma_{\f n} (\nabla \f d -\nabla \dd), \f \gamma_0( \f d - \nabla \f d ) \rangle _{\Hrand{-1},\Hrand{1}}\\  &+ \langle \f \gamma_{\f n} ( \f d - \dd), \f \gamma_0( \di \f d - \di \f d ) \rangle _{\Hrand{1},\Hrand{-1}}  \,.
\end{align*}
The boundary terms can be estimated by (compare~\cite[p.96 ff.]{mclean} for the definition of the Sobolev spaces on the boundary and the dual pairings)
\begin{align*}
& \langle \f \gamma_{\f n} (\nabla \f d -\nabla \dd), \f \gamma_0( \f d - \nabla \dd ) \rangle _{\Hrand{-1},\Hrand{1}}+ \langle \f \gamma_{\f n} ( \f d - \dd), \f \gamma_0( \di \f d - \di \f d ) \rangle _{\Hrand{1},\Hrand{-1}}  \\  &\qquad \leq  c \| \nabla \f d - \nabla \dd \|_{\Hrand{-1}}\| \f d - \dd \|_{\Hrand{1}}  
+ c \| \f d - \dd \|_{\Hrand{1}}\| \nabla\cdot \f d - \nabla\cdot \dd \|_{\Hrand{-1}} 
\\
&
 \qquad
 \leq c \| \f d - \dd \|_{\Hrand{1}}^2 \,.
\end{align*}
The above inequality  provides the assertion of this remark. 
\end{remark}

\subsection{Electromagnetic field effects}
In this section, we extend our model by an electromagnetic field influencing the dynamics of the liquid crystal. 
Therefore, the model is adapted by adding an electromagnetic potential to the free energy. 
The adapted free energy potential is given by~\cite[Section~3.2]{gennes}
\begin{align}
F_{\f H} (\f d ,\nabla \f d , \f H) = F( \f d ,\nabla \f d ) - \frac{\chi_{\|}}{2} | \f d \cdot \f H |^2 - \frac{\chi_{\bot}}{2} | \f d \times \f H|^2 \,,\label{electroenergy}
\end{align}
where the free energy potential $F$ is given by~\eqref{frei}. 
The associated variational derivative is given via the definition~\eqref{qdefq} by
\begin{align}
\f q _{\f H}  = \f q - \chi_{\|} ( \f d \cdot \f H) \f H + \chi_{\bot}\f H \times (\f H \times \f d) \,\label{qH}
\end{align}
with $\f q$ defined in~\eqref{qdef}. 
Here $ \chi_{\|} $ and $ \chi_{\bot}$ denote the constants measuring the magnetic susceptibility parallel and orthorgonal to the director,  both constants are negative $\chi_{\|} , \chi_{\bot} < 0$ due to the diamagnetic properties of liquid crystals (compare~\cite[Section~3.2]{gennes}). 
In usual nematic liquid crystals, it holds that $\chi_{\|}> \chi_{\bot}$, such that $|\chi_{\|}|< |\chi_{\bot}|$ which agrees with the naive perception since molecules that are not aligned should experience a bigger impact.
Using the calculation rules in Section~\ref{sec:not}, we observe for $\f d $ with $|\f d|=1$ that 
\begin{align*}
- \frac{\chi_{\|}}{2} | \f d \cdot \f H |^2 - \frac{\chi_{\bot}}{2} | \f d \times \f H|^2  = - \frac{\chi_{\|}}{2} | \f d \cdot \f H |^2 - \frac{\chi_{\bot}}{2} ( | \f d |^2 | \f H|^2 -  (\f d \cdot \f H)^2 )  = -  \frac{\chi_{\|}-\chi_{\bot}}{2} (\f d \cdot \f H)^2 - \frac{\chi_{\bot}}{2}  | \f H|^2\,.
\end{align*}
Since $\f H$ is given and $ {\chi_{\|}-\chi_{\bot}}>0$, this energy part is minimized if $\f d$ is parallel to $\f H$. 
Clearly, the molecules are then aligned along the direction of the electromagnetic field $\f H$.

\begin{proposition}[Integration-by-parts formula]\label{prop:intH}
Consider  $\f d $ and $\dd$ fulfilling~\eqref{reldiss}  and~\eqref{regtest}, respectively. Let additionally $\f H\in  \f L^6  $ and $\HH\in \f L^\infty$ be constant in time. Then 
\begin{align*}
& \chi_{\|}  \left (\f d(t) \cdot \f H , \dd (t)\cdot \HH \right ) + \chi_{\bot} \left ( \f d(t) \times \f H , \dd (t)\times \HH \right ) \\
&\leq    \chi_{\|}  \left (\f d(0) \cdot \f H , \dd (0)\cdot \HH \right ) + \chi_{\bot} \left ( \f d(0) \times \f H , \dd (0)\times \HH \right )\\
&\quad +\int_0^t \left [\left ( \f d \times \t \f d , \chi_{\|}\dd \times \HH (\dd \cdot \HH )- \chi_{\bot}\rot \dd \rot \HH  \rot \HH  \dd \right )  + \left (\dd \times \t \dd ,\chi_{\|} \f d \times \f H ( \f d \cdot \f H) - \chi_{\bot}\rot{\f d}\rot{ \f H}\rot{ \f H}  \f d \right ) \right ]
 \\
 & \quad 
 + \int_0^t 
 \|  \dd \times \t \dd \|_{\f L^\infty }\left ( (\chi_{\|}+\chi_{\bot})\| \dd \times \HH - \f d \times \f H \|_{\f L^2} \| \f d \cdot \f H - \dd \cdot \HH \|_{\Le} +\chi_{\bot} \| \HH\|_{\f L^3}\| \dd \|_{\f L^\infty} \| \f H- \HH \|_{\Le}  \| \f d - \dd \|_{\f L^{6}} \right )  \de s  
 \\
 &\quad + \delta \inttet{\mathcal{W}} +  \inttet{\mathcal{K} \mathcal{E}} + \inttet{\| \f H \|_{\f L^3} ^2 \| \HH\|_{\f L^\infty} ^2 \|\f d - \dd\|_{\f L^6}^2  \mathcal{K}} 
\\
&\quad  +  \inttet{\| \HH\|_{\f L^\infty} ^2 \mathcal{K} \left ( - \chi_{\|} \| \f d \cdot \f H - \dd \cdot \HH\|_{\Le}^2 - \chi_{\bot} \| \f d \times \f H - \dd \times \HH\|_{\Le}^2 \right ) }  
 \\ &\quad - \inttet{\left ( \mathcal{A}_2 , \chi_{\|} ( \f d \times \f H- \dd \times \HH ) ( \dd \cdot \HH)  - \chi_{\bot} ( \f d \cdot \f H - \dd \cdot \HH ) \HH \times \dd + \frac{\chi_{\bot} }{2} \f H ( \f d - \dd ) \rot{\HH}\dd\right ) }\,.
\end{align*}
holds for every $t\in [0,T]$. The dependence on $s$ of the terms under the integral is not written out for brevity. 
\end{proposition}
\begin{proof}
The fundamental theorem of calculus grants that
\begin{align*}
& \chi_{\|}  \left (\f d(t) \cdot \f H , \dd (t)\cdot \HH \right ) + \chi_{\bot} \left ( \f d(t) \times \f H , \dd (t)\times \HH \right ) \\
&=    \chi_{\|}  \left (\f d(0) \cdot \f H , \dd (0)\cdot \HH \right ) + \chi_{\bot} \left ( \f d(0) \times \f H , \dd (0)\times \HH \right )\\
&\quad +\int_0^t \left [\left (  \t \f d , \chi_{\|}\f H (\dd \cdot \HH )- \chi_{\bot}\f H \times( \HH \times \dd) \right )  + \left (\t \dd ,\chi_{\|}\HH ( \f d \cdot \f H) -  \chi_{\bot}\HH \times
( \f H \times \f d) \right ) \right ]\de s\,.
\end{align*}
Since $|\f d | = |\dd|=1$, we find with~\cite[Lemma~3.2]{weakstrong} that 
\begin{align*}
\int_0^t \left (  \t \f d , \chi_{\|}\f H (\dd \cdot \HH )-\chi_{\bot} \f H \times (\HH \times \dd ) \right )  \de s
 ={}& \int_0^t \left ( \f d \times  \t \f d , \chi_{\|}\f d \times (\f H (\dd \cdot \HH )- \chi_{\bot} \f H \times( \HH \times \dd)) \right )  \de s
\intertext{and}
\int_0^t \left (\t \dd ,\chi_{\|}\HH ( \f d \cdot \f H) - \chi_{\bot} \HH \times
 (\f H \times \f d) \right )\de s
 ={}& \int_0^t \left (\dd \times \t \dd ,\dd \times (\chi_{\|}\HH ( \f d \cdot \f H) - \chi_{\bot} \HH \times
 (\f H \times \f d)) \right ) \de s\,.
\end{align*}
 Simple rearrangements show that
 \begin{subequations}
  \begin{align}
 & \int_0^t \left [\left ( \f d \times  \t \f d , \f d \times (\chi_{\|}\f H (\dd \cdot \HH )-  \chi_{\bot}\f H \times ( \HH \times \dd) ) \right )  + \left (\dd \times \t \dd ,\dd \times (\chi_{\|}\HH ( \f d \cdot \f H) - \chi_{\bot} \HH \times
( \f H \times \f d)) \right ) \right ]\de s\notag
 \\
 &= \int_0^t \left [\left ( \f d \times \t \f d , \chi_{\|}\dd \times \HH (\dd \cdot \HH )-  \chi_{\bot}  \rot \dd \rot \HH\rot \HH  \dd \right )  + \left (\dd \times \t \dd , \chi_{\|}\f d \times \f H ( \f d \cdot \f H) - \chi_{\bot} \rot {\f d} \rot{ \f H} \rot{ \f H }\f d \right ) \right ]\de s
\notag \\
 &\quad + \int_0^t  \left (   \f d \times \t \f d  - \dd \times \t \dd  ,\chi_{\|} ( \f d \times \f H- \dd \times \HH ) ( \dd \cdot \HH) -  \chi_{\bot} ( \rot {\f d} \rot{ \f H}- \rot \dd \rot \HH ) \rot \HH  \dd    \right )\de s \label{line1}
 \\
 &\quad  + \int_0^t 
  \left (  \dd \times \t \dd , \chi_{\|}( \dd \times \HH - \f d \times \f H ) ( \f d \cdot \f H - \dd \cdot \HH )  -  \chi_{\bot}(\rot \dd\rot  \HH -\rot{ \f d} \rot{ \f H} ) ( \rot  {\f H} \f d   -\rot  \HH \dd ) \right )\de s    \,.\label{line2}
 \end{align}
  \end{subequations}

Estimating the terms in line~\eqref{line2} gives
\begin{align*}
&  \int_0^t 
  \left (  \dd \times \t \dd , \chi_{\|}( \dd \times \HH - \f d \times \f H ) ( \f d \cdot \f H - \dd \cdot \HH ) - \chi_{\bot} (\rot \dd\rot  \HH -\rot{ \f d} \rot{ \f H} ) ( \rot  {\f H} \f d   -\rot  \HH \dd ) \right ) \de s 
  \\&\leq  \int_0^t 
 \|  \dd \times \t \dd \|_{\f L^\infty }\left ((\chi_{\|}+\chi_{\bot}) \| \dd \times \HH - \f d \times \f H \|_{\f L^2} \| \f d \cdot \f H - \dd \cdot \HH \|_{\Le} + 2 \chi_{\bot}\| \HH\|_{\f L^3}\| \dd \|_{\f L^\infty} \| \f H- \HH \|_{\Le}  \| \f d - \dd \|_{\f L^{6}} \right )   \de s  
 \,,
\end{align*}
where we used 
\begin{align*}
(\rot \dd\rot  \HH -\rot{ \f d} \rot{ \f H} ) ( \rot  {\f H} \f d   -\rot  \HH \dd )  ={}& ( ( \dd \cdot \HH) I - \HH \o \dd - ( \f d \cdot \f H ) I + \f H \o \f d ) ( \rot  {\f H} \f d   -\rot  \HH \dd )\\ = {}& ( \dd \cdot \HH- \f d \cdot \f H )   ( \rot  {\f H} \f d   -\rot  \HH \dd ) - \HH \dd \rot {\f H} \f d - \f H \f d \rot \HH \dd 
\end{align*}
and 
\begin{align*}
-\HH \dd \rot {\f H} \f d - \f H \f d \rot \HH \dd ={}&- \HH \dd  \rot {\f H- \HH }\f d  + (\HH - \f H) \f d \rot\HH \dd\\ = {}&- \HH \dd  \rot {\f H- \HH }(\f d- \dd)  + (\HH - \f H) (\f d- \dd) \rot\HH \dd \,,
\end{align*}
which we find due to the skew-symmetry of $\rot{\f H}$ and $\rot \HH$. 

We use the symbol $\rot{}$ to write the cross-product, since it is a matrix multiplication and thus, in contrast to the cross-product associative.

Very similar to Lemma~\ref{lem:timederi}, we estimate the terms in line~\eqref{line1}. 
We abbreviate $\f a := \chi_{\|} ( \f d \times \f H- \dd \times \HH ) ( \dd \cdot \HH) -  \chi_{\bot} ( \rot {\f d} \rot{ \f H}- \rot \dd \rot \HH ) \rot \HH  \dd  $.
Adapting Lemma~\ref{lem:timederi}, we find for the terms in line~\eqref{line1} by inserting equation~\eqref{eq:mdir} for $\f d $  that 
\begin{align*}
\int_0^t \left ( \f d \times \t \f d -\dd \times \t \dd  , \f a\right ) \de s 
={}& -\inttet { \left ( \f d \times ( \f v \cdot \nabla ) \f d - \dd \times ( \vv\cdot \nabla) \dd - \f d \times  \sk v \f d - \dd \times \skv \dd , \f a\right ) }
\\
&- \inttet{\left ( \lambda \f d \times \sy v\f d -\lambda \dd \times \syv \dd + \f d \times \f q - \dd \times \tq, \f a 
\right )
} - \inttet{\left ( \mathcal{A}_2(s), \f a \right )  }\,.
\end{align*}
Similar rearrangements and estimates as in Lemma~\ref{lem:timederi} show
\begin{align*}
\int_0^t & { \left ( \f d \times ( \f v \cdot \nabla ) \f d - \dd \times ( \vv\cdot \nabla) \dd - \f d \times  \sk v \f d - \dd \times \skv \dd , \f a\right ) } \de s \\
\leq{}&  \delta \inttet{\| \f v - \vv \|_{\f L^6}^2 } + C_\delta \| \f d \|_{L^\infty(\f L^\infty)} ^2\inttet{  \left ( \ll{\nu_t, | \f S - \nabla \dd |^2 } \| \f d \|_{L^\infty(\f L^\infty)} ^2 \| \f a \|_{\f L^3}^2 + \| \nabla \dd \|_{L^\infty(\f L^3)}^2 \| \f a \|_{\f L^2}^2  \right )} \\
&+ \inttet{  \left [ \| \f d \|_{L^\infty(\f L^\infty)}\| \vv \|_{\f L^\infty} \ll{\nu_t | \f S - \nabla \dd |} \| \f a \|_{\f L^2}+ \| \f d -\dd \|_{\f L^6} \| \vv \|_{\f L^\infty} \| \nabla \dd \|_{\f L^3} \| \f a \|_{\f L^2}\right ]} \\
& + \delta \inttet{\| \sk v-\skv \|_{\Le}^2 } +  C_{\delta} \inttet{ \left [\| \f d \|_{L^\infty(\f L^\infty)}^4 \| \f a \|_{\f L^2}^2 + \| \f d - \dd \|_{\f L^6} (\| \f d \|_{L^\infty(\f L^\infty)} + \| \dd \|_{L^\infty(\f L^\infty)} ) \| \f a \|_{\f L^2}\right ]}\,
\end{align*}
as well as
\begin{align*}
 \int_0^t &{\left ( \lambda \f d \times \sy v\f d -\lambda \dd \times \syv \dd + \f d \times \f q - \dd \times \tq, \f a 
\right )
} \de s 
\\
\leq{}& \delta \inttet{\| \sy v\f d -\syv \dd \|_{\Le}^2 } + \inttet{\left [C_{\delta} \| \f d \|_{L^\infty( L^\infty)}^2 \| \f a \|_{\Le}^2 + \| \f d - \dd \|_{\f L^6} \| \syv \dd \|_{\f L^3} \| \f a \|_{\f L^2} \right ]} \\&+ \delta\inttet{ \| \f d \times \f q - \dd \times \tq \|_{\Le}^2} + C_{\delta } \inttet{\| \f a \|_{\Le}^2 } \,.
\end{align*}
For the abbreviation $\f a$, we find with the properties of the operator $\rot{}$ (see Section~\ref{sec:not}) that
\begin{align*}
\f a ={}&  \chi_{\|} ( \f d \times \f H- \dd \times \HH ) ( \dd \cdot \HH)  - \chi_{\bot} ( \f d \cdot \f H - \dd \cdot \HH ) \HH \times \dd + \chi_{\bot} \left ( \f H \o \f d - \HH \o \dd\right ) \rot{\HH} \dd \\
={}& \chi_{\|} ( \f d \times \f H- \dd \times \HH ) ( \dd \cdot \HH)  - \chi_{\bot} ( \f d \cdot \f H - \dd \cdot \HH ) \HH \times \dd + {\chi_{\bot} } \f H ( \f d - \dd ) \rot{\HH}\dd  \,.
\end{align*}

Due to the asserted regularity, i.e., $\f H \in \f L^3 $ as well as $\HH\in \f L^\infty$ and $\f d $, $\dd\in L^\infty(0,T;\f L^\infty)$, we observe that $\f a \in L^\infty(0,T;\f L^3) $ and the estimate
 \begin{align*}
\| \f a \|_{\Le} \leq {}&- \chi_{\|}  \| \HH\|_{\f L^\infty}\| \dd \|_{L^\infty(\f L^\infty)} \| \f d \times \f H - \dd \times \HH\|_{\Le} - \chi_{\bot}  \| \HH\|_{\f L^\infty}\| \dd \|_{L^\infty(\f L^\infty)} \| \f d \cdot \f H - \dd \cdot \HH\|_{\Le} \\&-  \chi_{\bot}  \| \HH\|_{\f L^\infty}\| \dd \|_{L^\infty(\f L^\infty)} \| \f H \|_{\f L^3} \| \f d -\dd \|_{\f L^6}\,.
\end{align*}
Note that $\chi_{\|}$ and $\chi_{\bot}$ are negative constants. 
Putting everything together yields the assertion. 
\end{proof}
We define the relative energy for the case of different electromagnetic fields via
\begin{subequations}\label{Nummer}
\begin{align}
\mathcal{E}^M_{\f H , \HH}(t) = \mathcal{E}^M(t) - \frac{\chi_{\|}}{2} \| \f d(t) \cdot \f H - \dd (t)\cdot \HH \|_{\Le}^2 - \frac{\chi_{\bot}}{2} \| \f d (t) \times \f H  - \dd (t) \times \HH  \|_{\Le}^2 \,
\end{align}
and the adapted function $\mathcal{K}_{\f H , \HH} $ via 
\begin{align}
\mathcal{K}_{\f H , \HH } :=( 1 + \| \HH \|_{\f L^\infty} ^2 + \| \f  H\|_{\f L^3} ^2 \| \HH \|_{\f L^\infty}^2 )  \mathcal{K}\,.
\end{align}
\end{subequations}

Note that $\f H$ and $\HH$ are constant in time. To be precise, we remark that also $\mathcal{A}$ and $\mathcal{W}$	 has to be adapted since we insert everywhere  $\tq_{\HH} $ instead of  $\tq$ and $\f q_{\f H}$ instead of $\f q$. The term $\tq_{\HH}$ is given by~\eqref{qH} with $\f q $, $\f d$, and $\f H$ replaced by $\tq$, $\dd$, and $\HH$, respectively.

With this, we prove an adapted relative energy inequality.

\subsection{Adapted relative energy inequality}
\begin{proposition}[Generalized relative energy inequality]\label{prop:bound}
Let $(\f v , \f d , (\nu^o,m,\nu^\infty),( \mu_t,\nu^\mu))$ be a measure-valued solution (see Definition~\ref{def:meas}) and let $(\vv, \dd)$ be a test function fulfilling the regularity assumptions~\eqref{regtest} but possibly different boundary conditions. 
Let $\f H \in \f L^3$ and $\HH \in \f L^\infty$ be given as constant functions in time.  
Then the adapted relative energy inequality
\begin{align}
\begin{split}
\frac{1}{2}\mathcal{E} ^M_{\f H, \HH  }(t) + {}&\frac{1}{2}\int_0^t\mathcal{W}(s) \exp\left ({\int_s^t\mathcal{K}_{\f H , \HH }(\tau)\de \tau }\right )\de s
+ \left ( 1 - \exp\left ({\int_0^t\mathcal{K}_{\f H , \HH }(s)\de s }\right )\right )( \| \f d - \dd \|_{\Hrand{1}}^2 + \| \f H - \HH \|_{\f L^2}^2 ) 
\\ \leq{}&  \Big ( \mathcal{E}_{\f H, \HH  }(0)+\frac{|\f \Theta|^2}{2k}\|\nabla \dd \o \dd \|_{L^\infty (L^\infty)} ^2 \left \|   \f d(0)- \dd(0)    \right \|_{\Le}^2 \Big) \exp\left ({\int_0^t\mathcal{K}_{\f H , \HH }(s)\de s } \right )
\\
& +\Big ((\nabla \f d(0)-\nabla \dd (0))\o (\f d(0)-  \dd(0)) \dreidotkom \f \Theta \dreidots \nabla \dd (0)\o \dd(0) \Big) \exp\left ({\int_0^t\mathcal{K}_{\f H , \HH }(s)\de s }\right ) 
\\&
+ \int_0^t \left ( \mathcal{A}(s) , 
\begin{pmatrix}
\vv -\f v \\
\f d \times (\tq_{\HH} -   \f q_{\f H} )  + \f d   \times \dd \frac{|\f \Theta|^2}{k}\|\nabla \dd \o \dd \|_{L^\infty(\f L^\infty)} ^2   - \f a_{\f H , \HH} 
\end{pmatrix}\right )
\exp\left ({\int_s^t\mathcal{K}_{\f H , \HH }(\tau)\de \tau }\right )\de s  
\, 
\end{split}\label{relenineqbound}
\end{align} 
holds for almost all $t\in (0,T)$, where $\mathcal{A} $ is given by~\eqref{A} and $\f a_{\f H, \HH} $  by
\begin{align*}
\f a_{\f H, \HH} := \chi_{\|} ( \f d \times \f H- \dd \times \HH ) ( \dd \cdot \HH)  - \chi_{\bot} ( \f d \cdot \f H - \dd \cdot \HH ) \HH \times \dd + {\chi_{\bot} } \f H ( \f d - \dd ) \rot{\HH}\dd\,.
\end{align*}
The potentials $\mathcal{E}_{\f H , \HH } $ and $\mathcal{K}_{\f H , \HH }$ are defined by~\eqref{Nummer}.

\end{proposition}
\begin{remark}
This proposition immidaiately generalizes the weak-strong uniqueness result for measure-valued solutions~\cite{weakstrong} and for dissipative solutions (see Remark~\ref{rem:weakstrong}) to a system incorporating the influence of an electromagnetic field. 
Additionally, this relative energy inequality  provides a result on the continuous dependence of solutions on the difference in the electromagnetic field, the boundary values and initial values, as long as a regular solution fulfilling~\eqref{regtest} exists. 

We incorporate this additional electromagnetic field in order to use it as a control parameter in a future article (see~\cite{approx}).
\end{remark}
\begin{proof}
The only difference of the above inequality in comparison to the relative energy inequality~\eqref{relEnergyinequ} in the case of equal boundary values and electro-magnetic field effects are the additional terms in the first line and the additional term in the last line, respectively.  

Going through the proof of Theorem~\ref{thm:main} and inserting every time Lemma~\eqref{Sobolev}  and Corolllary~\ref{cor:rand} instead of~\cite[Lemma~3.1]{weakstrong} yields together with Proposition~\ref{prop:intH} the estimate
\begin{align*}
\begin{split}
\frac{1}{2}\mathcal{E}^M_{\f H, \HH} (t)& + \int_0^t\mathcal{W}(s) \de s 
\\
\leq{}& \mathcal{E}_{\f H, \HH} (0)  + \delta\int_0^t\mathcal{W}(s) \de s +  \int_0^t\mathcal{K}_{\f H, \HH} (s) \mathcal{E}_{\f H,\HH} ^M(s)\de s  \\
&
+\left ((\nabla \f d(0)-\nabla \dd (0))\o (\f d(0)-  \dd(0))\dreidotkom \f \Theta \dreidots \nabla \dd (0)\o \dd(0)\right ) + \frac{|\f \Theta|^2}{k}\| \nabla \dd \o \dd \|_{L^\infty( \f L^\infty)}^2 \left \|   \f d(0)- \dd(0)    \right \|_{\Le}^2 
\\
& + \inttet{ \left (\mathcal{A}(s), \begin{pmatrix}
 \vv- \f v  \\ \f d \times (\tq _{ \HH}-   \f q_{\f H}  ) + \frac{2|\f \Theta|^2}{k}\| \nabla \dd \o \dd \|_{L^\infty( \f L^\infty)}^2 \f d \times \dd - \f a_{ \f H , \HH}  
 \end{pmatrix}\right )}\\
 &+ \inttet{ \mathcal{K}_{ \f H, \HH}(s)\left ( \| \tr( \f d ) - \tr (\dd)\|_{\Hrand{1}}^2  + \| \f H - \HH\|_{\Le}^2 \right )
 }
\,. 
\end{split}
\end{align*}
From Gronwall's estimate, we may infer
\begin{align}
\begin{split}
\frac{1}{2}\mathcal{E}^M_{\f H,\HH}  (t) + {}&\frac{1}{2}\int_0^t\mathcal{W}(s) \exp\left ({\int_s^t\mathcal{K}_{\f H, \HH} (\tau)\de \tau }\right )\de s
\\ \leq{}&  \Big ( \mathcal{E}_{\f H,\HH }(0)+\frac{|\f \Theta|^2}{k}\|\nabla \dd \o \dd \|_{L^\infty (L^\infty)} ^2 \left \|   \f d(0)- \dd(0)    \right \|_{\Le}^2 \Big) \exp\left ({\int_0^t\mathcal{K}_{ \f H, \HH}(s)\de s } \right )
\\
& +\Big ((\nabla \f d(0)-\nabla \dd (0))\o (\f d(0)-  \dd(0)) \dreidotkom \f \Theta \dreidots \nabla \dd (0)\o \dd(0) \Big) \exp\left ({\int_0^t\mathcal{K}_{ \f H, \HH}(s)\de s }\right ) 
\\&
+ \int_0^t \left ( \mathcal{A}(s) , 
\begin{pmatrix}
\vv -\f v \\
\f d \times (\tq_{\HH} -   \f q_{\f H} )  + \f d   \times \dd \frac{2|\f \Theta|^2}{k}\|\nabla \dd \o \dd \|_{L^\infty(\f L^\infty)} ^2   - \f a_{ \f H , \HH} 
\end{pmatrix}\right )
\exp\left ({\int_s^t\mathcal{K}_{ \f H, \HH}(\tau)\de \tau }\right )\de s  
\\&+ \inttet{ \mathcal{K}_{ \f H, \HH}(s) \exp\left ({\int_s^t\mathcal{K}_{ \f H, \HH}(\tau)\de \tau }\right ) \left ( \| \tr( \f d ) - \tr (\dd)\|_{\Hrand{1}}^2  + \| \f H - \HH\|_{\Le}^2 \right  )
 }\,.
 \end{split}\label{einederer}
 \end{align} 
 With an integration-by-parts, we observe for the term in the last line
 \begin{align*}
  \int_0^t& \mathcal{K}_{ \f H, \HH}(s) \exp\left ({\int_s^t\mathcal{K}_{ \f H, \HH}(\tau)\de \tau }\right )  \left ( \| \tr( \f d ) - \tr (\dd)\|_{\Hrand{1}}^2  + \| \f H - \HH\|_{\Le}^2 \right  )\de s 
  \\= & \left [  - \exp\left ({\int_s^t\mathcal{K}_{ \f H, \HH}(\tau)\de \tau }\right )   \left ( \| \tr( \f d ) - \tr (\dd)\|_{\Hrand{1}}^2  + \| \f H - \HH\|_{\Le}^2 \right  ) \right ]_{s=0}^t \\ & \quad  + \inttet{  \exp\left ({\int_s^t\mathcal{K}_{ \f H, \HH}(\tau)\de \tau }\right ) \frac{\partial}{\partial s } \left ( \| \tr( \f d ) - \tr (\dd)\|_{\Hrand{1}}^2  + \| \f H - \HH\|_{\Le}^2 \right  )
 }
 \\ =  &\left (  \exp\left ({\int_0^t\mathcal{K}_{ \f H, \HH}(s)\de s }\right )- 1 \right )  \left ( \| \tr( \f d ) - \tr (\dd)\|_{\Hrand{1}}^2  + \| \f H - \HH\|_{\Le}^2 \right  ) \,.
 \end{align*}
Note that the second equality holds since the boundary conditions and the electromagnetic field effects are constant in time. 
Inserting the integration-by-parts formula into~\eqref{einederer} proves the assertion. 

\end{proof}

\section{Long-time behavior\label{sec:6}}
In this section, we are focusing on the long time behavior of the measure-valued as well as the dissipative solutions. 
Again, the relative energy inequality will be essential and we discuss its implications on the long-time behavior of solutions. 
For this section, we assume again $\f H = \HH = 0$, i.e., that there is no electromagnetic field acting on the liquid crystal. All arguments also hold for an additional electromagnetic field (see Remark~\ref{rem:elec}) , but we exclude this possibility to keep the arguments accessible. 

\subsection{Necessary first order  optimality conditions for the Oseen--Frank energy}

First, we recall the Euler--Lagrange equation for the Oseen--Frank energy.
\begin{corollary}\label{cor:Euler}
The Euler-Lagrange equation for a minimizer of the Oseen--Frank energy can be expressed via
\begin{align*}
0 ={}& \f d \times \f q =  - \di \left (\f d \times \left (\f \Lambda : \nabla \f d + \f d \cdot \f \Theta \dreidots \nabla \f d \o \f d \right )\right ) +\nabla \rot{\f d }  : \left (\f \Lambda : \nabla \f d + \f d \cdot \f \Theta \dreidots \nabla \f d \o \f d \right ) + \f d \times(  \nabla \f d: \f \Theta \dreidots \nabla \f d \o \f d    )
\\
= {}&-\di \left (\f d \times F_{\f S} ( \f d , \nabla \f d) \right ) + \nabla \rot{\f d} : F_{\f S}(\f d ,\nabla \f d) + \f d\times F_{\f h}(\f d , \nabla \f d)  
 \,
\end{align*}
for $\f d$ regular enough, i.e., $\f d \in \Hc$. 
\end{corollary}
\begin{proof}
A minimizer of the Oseen--Frank energy is a stationary point and hence, fulfills the associated first order optimality condition, the Euler--Lagrange equations. 
These are given in Hardt, Kinderlehrer and Lin~\cite{hardtlin}, see also~\cite{hong}.
We omit the derivation here. In our Notation, the Euler--Lagrange equation  for a minimizer $\f d$ of the Oseen--Frank energy can be expressed via (compare~\cite[Section~1]{hardtlin})
\begin{multline}
- \di \left (  \pat{F}{\f S} (\f d , \nabla \f d ))- \f d \o \f d \pat{F}{\f S } ( \f d , \nabla \f d)   \right ) + \pat{F}{\f h }(\f d , \nabla \f d)  -\left (\pat{F}{\f h }(\f d , \nabla \f d \right ) \cdot \f d) \f d  \\ - \left ( \pat{F}{\f S} ( \f d , \nabla \f d) : \nabla \f d\right ) \f d - \nabla \f d \left ( \pat{F}{\f S}(\f d , \nabla \f d)\right )^T \f d = 0 \,.\label{eulerLagrange}
\end{multline}
Calculating the divergence of the second term on the left-hand side gives 
\begin{align*}
\di\left ( \f d \o \f d \pat{F}{\f S } ( \f d , \nabla \f d) \right ) = \f d \o \f d \left ( \di  \pat{F}{\f S }\right )  + \left ( \pat{F}{\f S} ( \f d , \nabla \f d) : \nabla \f d\right ) \f d + \nabla \f d \left ( \pat{F}{\f S}(\f d , \nabla \f d)\right )^T \f d 
\end{align*}
Note that the second and third term on the right-hand side of the previous equation cancel with the terms in the second line of equation~\eqref{eulerLagrange}. 
Thus, equation~\eqref{eulerLagrange} can be simplified to
\begin{align*}
0 =( I - \f d \o \f d) \left (  - \di \pat{F}{\f S}(\f d ,\nabla \f d ) + \pat{F}{\f h} ( \f d ,\nabla \f d)   \right ) =  ( I - \f d \o \f d) \f q \,.
\end{align*}
Where $\f q$ is the variational derivative defined in~\eqref{qdefq}.
Since $\f d$ is a unit vector, we can use the calculus for the cross product introduced in Section~\ref{sec:not} and observe
\begin{align*}
 I - \f d \o \f d =  | \f d |^2 I - \f d \o \f d = \rot{\f d}^T \rot{\f d}\,. 
\end{align*}
Due to the properties of the cross product, a simple observation holds:
For a unit vector $ \f d \in \S2 $ and for all vectors $ \f a \in \R^3$ it holds  that $ \f d \times \f a = 0 \, \Leftrightarrow \, \f d \times \f d \times \f a =0$. 
Indeed, the first implication is obvious, whereas the second implication follows from the calculation rules for the cross product by taking the equation again in the cross product with $\f d$, i.e.,
\begin{align*}
\f d \times \f d \times \f d \times \f a = \f d \times \left ( | \f d|^2 I -\f  d \o \f d \right ) \f a = | \f d|^2 \f d \times \f a - \f d \times \f d ( \f d \cdot \f a) = \f d \times \f a \,.
\end{align*}
Hence, the Euler--Lagrange equation for minimizers of the Oseen--Frank energy can be expressed as stated in the assertion of Corollary~\ref{cor:Euler}.
\end{proof}

\subsection{Convergence for large times}
In the following we consider constant test functions $(\vv , \dd)$. Due to the homogeneous Dirichlet boundary conditions for the velocity field, it is obvious that $\vv \equiv 0$. For the director field we assume constant boundary conditions $ \tr(\dd)= \f d _0 \in \S2$. 
For this system the constant solution $(\vv, \dd)= (0, \f d_0)$ is an obvious solution to the system~\eqref{eq:strong}.

\begin{theorem}
\label{thm:longtime}
Consider the Ericksen--Leslie system equipped with the Oseen--Frank energy~\eqref{eq:strong}-\eqref{frei} 
with vanishing right-hand side, i.e.,~$\f g \equiv 0$.
Then there exists a sequence of time steps $\{t_n\}_{n\in\N}$ with $t_n \ra \infty$ as $n\ra \infty$ such that the measure-valued solutions converge to a steady state, i.e., 
$ ( \f v (t_n) , \f d (t_n) )  \rightharpoonup  ( \f v_\infty , \f d_\infty) $ in $ \He \times \He$, where $ \f v_\infty\equiv 0$ and the generalized gradient Young measure 
$\nu_\infty$ solves the associated Euler--Lagrange equations for the Oseen--Frank energy in a measure-valued sense, i.e., 
\begin{align}
 \left (\rot{\f d_\infty}  F_{ \f S} ( \f d _\infty, \nabla \f d_\infty ) ; \nabla \f \psi  \right ) + \ll{\nu_\infty, \f \Upsilon :\left (\f  S    (F_{\f S}(\f h, \f S))^T\right ) \cdot\f \psi   }+{\ll{\nu_\infty, \left (\f h \times F_{\f h}(\f h, \f S)\right ) \cdot\f \psi   }}= 0 \,\label{EulerMEas}
\end{align}
for $\f \psi \in L^1 (0,T; \He \cap \C(\ov \Omega)$. 

Additionally, there exists a sequence of time steps $\{t_n\}_{n\in\N}$ with $t_n \ra \infty$ for $n\ra \infty$ such that the dissipative solutions converge to a steady state too, i.e., 
$ ( \f v (t_n) , \f d (t_n) , \f d (t_n)\times \f q(t_n))  \ra ( 0 , \f d_\infty, 0) $ in $ \He \times \He\times \Le$.
\end{theorem}
\begin{remark}
The equation~\eqref{EulerMEas} is the Euler--Lagrange equation for the Oseen--Frank energy derived in Corollary~\eqref{cor:Euler} in a measure-valued sense. 
This can be seen as the measure-valued generalization of the necessary first-order optimality condition of the Oseen-Frank energy. 

Note that for a convex energy (which the Oseen--Frank energy is not) this would immediately  imply that $\f d$ is a global minimizer.
\end{remark}

\begin{proof}

Inserting the constant solution $(\vv, \dd)= (0 ,\f d_0)$ as well as $ \f H = \HH = 0$  in the relative energy inequality~\eqref{relenineqbound} for different boundary values 
gives 
\begin{align}
\begin{split}
&\frac{1}{2}\mathcal{E}^M(t) +\int_0^t \mathcal{W}(s)\de s 
\\ &\quad \leq{}   \mathcal{E}(0)+\frac{|\f \Theta|^2}{k}\|\nabla \dd \o \dd \|_{L^\infty (L^\infty)} ^2 \left \|   \f d(0)- \dd(0)    \right \|_{\Le}^2
 +(\nabla \f d(0)-\nabla \dd (0))\o (\f d(0)-  \dd(0)) \dreidotkom \f \Theta \dreidots \nabla \dd (0)\o \dd(0)
\,.
\end{split}
\label{Lyapunov}
\end{align}
Note that $\syv\equiv 0$ and $\t \vv\equiv 0 $ as well as $\nabla \dd=0$ and $\t \dd=0$ such that the potential $\mathcal{K}$ in proposition~\ref{prop:bound} vanishes. 

The inequality~\eqref{Lyapunov} implies that the relative energy is a Lyapunov-kind function for the special  equilibrium solution~$(\vv, \dd)=(0,\f d_0)$. 
We remark that this holds for all measure-valued solutions (see Definition~\ref{def:meas}) and dissipative solutions (see Definition~\ref{def:diss}). 


Inequality~\eqref{Lyapunov} implies that the integral $\int_0^\infty\mathcal{W}(s)\de s$ converges. 
The definition of the relative dissipation~\eqref{relW} let us conclude that there exists a sequence $\{t_n\}_{n\in\N}$ with $t_n\ra \infty $ for $n \ra \infty$ such that
\begin{align}
 ( \nabla \f v (t_n))_{\sym}\ra 0  \text{ in } \Le \qquad \text{and} \qquad \f d (t_n) \times \f q(t_n) \ra 0 \text{ in } \Le\label{Wconver}
\end{align}
for $n\ra \infty$.
Due to the homogeneous Dirichlet boundary conditions prescribed for the velocity field, we may infer that $\f v(t_n ) \ra 0 $ for $n\ra \infty$. 
Additionally, we infer that the limit measure $\nu_{\infty} = \lim_{t_n \ra \infty} \nu_{t_n}$ fulfills the first order necessary condition for a minimizer of the Oseen--Frank energy, i.e., the equation
\begin{align*}
 (\rot{\f d_\infty}  F_{ \f S} ( \f d _\infty, \nabla \f d_\infty  ; \nabla \f \psi ) + \ll{\nu_\infty, \f \Upsilon :\left (\f  S    (F_{\f S}(\f h, \f S))^T\right ) \cdot\f \psi   }+{\ll{\nu_\infty, \left (\f h \times F_{\f h}(\f h, \f S)\right ) \cdot\f \psi   }}= 0 \,
\end{align*}
for all $\f \psi \in \Hb \cap \C(\ov \Omega) $. The limits are naturally defined via $\f d_\infty= \lim_{t_n\ra \infty}\f d(t_n)$. Note the difference in the notation $\nu_\infty $ is the generalized Young measure for $t_n\ra \infty$ and $\nu^\infty$ denotes the defect angle measure. 

To estimate the time derivative of $\f d(t_n)$, we employ~\eqref{eq:mdir} and the identity $ | \f d | ^2 \t \f d  - \frac{1}{2} \t | \f d|^2 \f d = - \f d \times \f d \times \t \f d $
\begin{align*}
 \| \t \f d\|_{\f L^{6/5}} ={}&  \left \| | \f d | ^2 \t \f d  - \frac{1}{2} \t | \f d|^2 \f d\right  \| _{\f L^{6/5}}= \| \f d \times \f d \times \t \f d \| _{\f L^{6/5}}
 \\ \leq{}& \| \f d \|_{\f L^{12}} \| \f d \times \t \f d \|_{\f L^{4/3}} \leq \| \f d \|_{\f L^{12}} \| \f d \times \left ( (\f v \cdot \nabla) \f d  -\sk v \f d + \lambda \sy v \f d + \f q   \right )\|_{\f L^{4/3}}\\ \leq{}& c\| \f d \|_{\f L^{12}} \left (\| \f d\|_{\f L^{12}}  \| \f v \|_{\f L^6} \| \nabla \f d\|_{\f L^2}+ \| \f d \|_{\f L^{12}} ^2 \| \sk v \|_{\f L^2}  + \lambda \| \f d \|_{\f L^{12}} ^2 \| \sy v \|_{\f L^2}  + \| \f d \times \f q \|_{\f L^2}\right ) \,.
\end{align*}
Due to~\eqref{Wconver}, the right hand side of the previous estimate converges to zero as $t_n\ra \infty$. 
This implies that $\f d_\infty $ is actually a steady state. 

In regard of~\eqref{relenineqbound} and Proposition~\ref{prop:relenab}, the inequality~\eqref{Lyapunov} also holds for dissipative solutions, with $\mathcal{E}^M $ replaced by $\mathcal{E}$. 
The previous considerations also imply the assertions for the case of dissipative solutions. 
\end{proof}
\begin{remark}[Selectivity of Euler--Lagrange equations in the measure-valued sense]
It is argued in Roub\'i\v{c}ek~\cite[Remark~5.3.8]{RoubicekMeasure} that measure-valued extensions of the Euler--Lagrange equation in the sense of~\eqref{EulerMEas} can loose dramatically in selectivity. 
This means that such generalizations can admit astonishing many solutions apparently without any physical motivation (see also Roub\'i\v{c}ek and Hoffmann~\cite{hoff}). 

This is similar to the natural question: What is the connection between a measure fulfilling the Euler--Lagrange equation~\eqref{EulerMEas} and a minimizer of the Oseen--Frank energy.

The remedy proposed in~\cite{RoubicekMeasure} and~\cite{hoff} relies on the strategy of first relaxing the minimization problem and then deriving first order optimality conditions. 
For a general minimizing problem for a functional of the type~\eqref{frei} with $F: \R^ 3 \times \R^{3\times 3} \ra \R$ , this would result in the relaxed minimizing problem 
\begin{align*}
\min_{\{\f d \in \f L^\infty , \nu_t \in \GY\}}\int_{\Omega} \ll{\nu_t,  F( \f d (\f x) , \f S ) }  \,,\qquad\text{with}\qquad 
 \ll{\nu_t , \f S - \nabla \f d } = 0 \quad \text{a.e.~in }\Omega \,,
\end{align*}
where $\GY$ denote the generalized Young measures (see Definition~\ref{def:meas}). 
First order optimality conditions for such a problem are derived in~\cite[Proposition~5.3.2]{RoubicekMeasure} and~\cite{Sensitivity} to be: There exists a $\f \lambda^* \in \f L^2$ such that 
\begin{subequations}
\begin{align}
\di \f \lambda ^* ={}& \int_{\R^{d\times d } }  F _{\f h} ( \f d , \f S) \nu^o(\de \f S ) \qquad & \text{in }\Hd \,,\label{el1}\\
\int_{\Omega} \int_{\R^{d\times d }} \left ( \f \lambda ^* : \f S - F ( \f d(\f x) , \f S )\right ) \nu^o(\de \f S ) \de \f x = {}& \sup _{\f R \in \f L^2 } \int_{\Omega} \left ( \f \lambda ^*(\f x)  : \f R(\f x) - F ( \f d (\f x) , \f R(\f x) )\right ) \de \f x \qquad & \text{in }\R \,.\label{el2}
\end{align}
\end{subequations}
Note that the assumptions in~\cite{Sensitivity} are not fulfilled for the Oseen--Frank energy~\eqref{frei}. Especially the growth conditions are violated. Nevertheless, the second conditions~\eqref{el2} holds (see~\cite[Lemma~3.1]{Sensitivity}).
This second condition can be seen as a generalization of the classical Weierstra\ss{} condition~\cite[Section 48.8]{zeidler3}
\begin{align*}
F( \f d(\f x) ,\f R) - F( \f d (\f x), \nabla \f d(\f x)) \geq F_{\f S} ( \f d(\f x) , \nabla \f d(\f x)) : ( \f R - \nabla \f d(\f x)) \,.
\end{align*}
which should hold for all $\f R \in \R^{3\times 3}$. 

In the following, we want to argue that, for the special case of the Oseen--Frank energy, the condition~\eqref{el2} is fulfilled for $\f \lambda^*(\f x) = \int_{\R^{3\times 3}}  F_{\f S}(\f d(\f x ), \f S) \nu^o_{\f x}(\f S)= F_{\f S}(\f d(\f x ) , \nabla \f d(\f x)) $. 

%
Note that the case of the Oseen--Frank energy is considerably more difficult than the energy considered in~\cite{Sensitivity} due to the additional norm restriction of the director $| \f d|=1$ and due to the possible concentration effects. 
%

First, we observe for the left-hand side of~\eqref{el2} that
\begin{align}\label{nuuu}
\begin{split}
\ll{\nu , F_{\f S} ( \f h , \f S ) : \f S - F( \f h, \f S)  }  ={}& \ll{\nu, \f S : \f \Lambda : \f S + \f S \o \f h \dreidots \f \Theta \dreidots \f S \o \f h - \frac{ 1}{2}\f S : \f \Lambda : \f S + \f S \o \f h \dreidots \f \Theta \dreidots \f S \o \f h  }
\\ 
={}& \frac{1}{2}\int_{\Omega} \int_{\R^{3\times3}} \left (\f S : \f \Lambda : \f S + \f S \o \f d(\f x) \dreidots \f \Theta \dreidots \f S \o \f d(\f x)   \right ) \nu^o_{\f x } (\de \f S) \de \f x 
\\
{}&+  \frac{1}{2}\int_{\ov \Omega} \int_{\R^{3\times3}} \left (\f S : \f \Lambda : \f S + \f S \o \f h \dreidots \f \Theta \dreidots \f S \o \f h   \right ) \nu^\infty_{\f x } (\de \f h,\de \f S) m(\de \f x) 
\end{split}
\end{align}
For the right-hand side of~\eqref{el2}, we find for all $ \f R\in \Le$ that 
\begin{align}\label{nuuu2}
\begin{split}
 \int_{\Omega}& \left (   F_{\f S} ( \f d(\f x) ,  \nabla \f d( \f x)  ) : \f R( \f x)  - F( \f d (\f x), \f R( \f x) ) \right ) \de \f x\\
 &=
\int_{\Omega} \left (  
\nabla \f d (\f x) : \f \Lambda : \f R(\f x) + \nabla \f d(\f x ) \o \f d(\f x) \dreidots \f \Theta \dreidots \f R (\f x) \o \f d(\f x) \right ) \de \f x \\ &  -\frac{1}{2}\int_{\Omega} \left (\f R( \f x) : \f \Lambda : \f R(\f x) -\f R(\f x) \o \f d(\f x) \dreidots \f \Theta \dreidots \f R(\f x) \o \f d(\f x) 
 \right ) \de \f x \,.
 \end{split}
\end{align}
Subtracting~\eqref{nuuu2} from~\eqref{nuuu} yields with~\eqref{identify}
\begin{align*}
&\ll{\nu , F_{\f S} ( \f h , \f S ) : \f S - F( \f h, \f S)  } - \int_{\Omega} \left (   F_{\f S} ( \f d(\f x) ,  \nabla \f d( \f x)  ) : \f R( \f x)  - F( \f d (\f x), \f R( \f x) ) \right ) \de \f x &
\\
&\quad ={}\frac{1}{2} \ll{\nu, ( \f S - \f R ) : \f \Lambda :( \f S - \f R )  +  ( \f S - \f R )\o \f h  \dreidots \f \Theta  \dreidots( \f S - \f R )\o \f h   }& \geq 0 \,.
\end{align*}
Note that the terms in the last line of~\eqref{nuuu}are positive. 
Therefore, the inequality 
\begin{align*}
\ll{\nu , F_{\f S} ( \f h , \f S ) : \f S - F( \f h, \f S)  } \geq  \int_{\Omega} \left (   F_{\f S} ( \f d(\f x) ,  \nabla \f d( \f x)  ) : \f R( \f x)  - F( \f d (\f x), \f R( \f x) ) \right ) \de \f x \,
\end{align*}
is valid for all $\f R \in \f L^2(\Omega ; \R^{3\times 3})$.
The definition~(see~\cite[Definition~2]{weakstrong}) of a generalized gradient Young measure implies that there exists a sequence $\{ \fn d \} _{n\in\N} \subset \He $ such that $ \{\nabla \fn d\}_{n\in\N} $ converges to $\nu$ in the sense of generalized gradient Young measures. Choosing now the sequence  $ \{\nabla \fn d\}_{n\in\N} $ as the function $\f R$ in~\eqref{nuuu2} and passing to the limit as $n\ra \infty$ leads to the assertion. 
 
 A relaxed Weierstrass-condition is thus intrinsically fulfilled for the problem at hand. 
\end{remark}

\begin{remark}[Electromagnetic field]\label{rem:elec}
The assertions of Corollary~\ref{cor:Euler} and Theorem~\ref{thm:longtime} remain true, if a temporarily constant electromagnetic field $\f H\in\f L^3$ is present. 
In the case of Theorem~\ref{thm:longtime}, this can be observed, by choosing $\HH= 0$, when applying~\eqref{relenineqbound}.
\end{remark}

\addcontentsline{toc}{section}{References}

\bibliographystyle{abbrv}

\begin{thebibliography}{10}

\bibitem{singul2}
G.~P. Alexander, B.~Chen, E.~A. Matsumoto, and R.~D. Kamien.
\newblock Disclination loops, point defects, and all that in nematic liquid
  crystals.
\newblock {\em Rev. Mod. Phys.}, 84:497--514, 2012.

\bibitem{raymond}
D.~{Ars{\'e}nio} and L.~{Saint-Raymond}.
\newblock {From the Vlasov-Maxwell-Boltzmann system to incompressible viscous
  electro-magneto-hydrodynamics}.
\newblock {\em ArXiv e-prints}, 2016.

\bibitem{prohl}
R.~Becker, X.~Feng, and A.~Prohl.
\newblock Finite element approximations of the {E}ricksen--{L}eslie model for
  nematic liquid crystal flow.
\newblock {\em SIAM J. Numer. Anal.}, 46(4):1704--1731, 2008.

\bibitem{allgemein}
C.~Cavaterra, E.~Rocca, and H.~Wu.
\newblock Global weak solution and blow-up criterion of the general
  {E}ricksen--{L}eslie system for nematic liquid crystal flows.
\newblock {\em J. Differential Equations}, 255(1):24--57, 2013.

\bibitem{dafermos}
C.~M. Dafermos.
\newblock The second law of thermodynamics and stability.
\newblock {\em Arch. Ration. Mech. Anal.}, 70(2):167--179, 1979.

\bibitem{dafermos2}
C.~M. Dafermos.
\newblock {\em Hyperbolic Conservation Laws in Continuum Physics}.
\newblock Springer, Berlin, 2016.

\bibitem{gennes}
P.~G. De~Gennes.
\newblock {\em The physics of liquid crystals}.
\newblock Clarendon Press, Oxford, 1974.

\bibitem{diestel}
J.~Diestel and J.~J. Uhl, Jr.
\newblock {\em Vector measures}.
\newblock Providence, Rhode Island, 1977.

\bibitem{edwards}
R.~E. Edwards.
\newblock {\em Functional analysis. {T}heory and applications}.
\newblock Holt, Rinehart and Winston, New York, 1965.

\bibitem{sabine}
E.~Emmrich, S.~H. Klapp, and R.~Lasarzik.
\newblock Nonstationary models for liquid crystals: A fresh mathematical
  perspective.
\newblock {\em J. Non-Newton. Fluid Mech.}, 259:32--47, 2018.

\bibitem{unsere}
E.~Emmrich and R.~Lasarzik.
\newblock Existence of weak solutions to the {E}ricksen--{L}eslie model for a
  general class of free energies.
\newblock {\em Math. Meth. Appl. Sci.}, 41(16):6492--6518, 2018.

\bibitem{feireislstab}
E.~Feireisl.
\newblock Relative entropies in thermodynamics of complete fluid systems.
\newblock {\em Discrete Contin. Dyn. Syst.}, 32(9):3059--3080, 2012.

\bibitem{feireislsingular}
E.~Feireisl.
\newblock Relative entropies, dissipative solutions, and singular limits of
  complete fluid systems.
\newblock In {\em Hyperbolic Problems: Theory, Numerics, Applications},
  volume~8 of {\em AIMS on Applied Mathematics}, pages 11--28. AIMS,
  Springfield, USA, 2014.

\bibitem{novotny}
E.~Feireisl and A.~Novotn{\'y}.
\newblock Weak-strong uniqueness property for the full
  {N}avier--{S}tokes--{F}ourier system.
\newblock {\em Arch. Ration. Mech. Anal.}, 204(2):683, 2012.

\bibitem{isothermal}
E.~Feireisl, E.~Rocca, and G.~Schimperna.
\newblock On a non-isothermal model for nematic liquid crystals.
\newblock {\em Nonlinearity}, 24(1):243--257, 2011.

\bibitem{furihata}
D.~Furihata and T.~Matsuo.
\newblock {\em Discrete variational derivative method: {A} structure-preserving
  Numerical method for partial differential equations}.
\newblock CRC Press, Oxford, 2010.

\bibitem{groeg}
H.~Gajewski, K.~Gr{\"o}ger, and K.~Zacharias.
\newblock {\em Nichtlineare Operatorgleichungen und
  Operatordifferential-Gleichungen}.
\newblock Akademie-Verlag, Berlin, 1974.

\bibitem{dissCamassa}
K.~{Grunert}, H.~{Holden}, and X.~{Raynaud}.
\newblock {Global dissipative solutions of the two-component Camassa-Holm
  system for initial data with nonvanishing asymptotics}.
\newblock {\em Nonlinear Anal. Real World Appl.}, 17:203 -- 244, 2014.

\bibitem{hardtlin}
R.~Hardt, D.~Kinderlehrer, and F.-H. Lin.
\newblock Existence and partial regularity of static liquid crystal
  configurations.
\newblock {\em Comm. Math. Phys.}, 105(4):547--570, 1986.

\bibitem{Pruess2}
M.~Hieber, M.~Nesensohn, J.~Pr{\"u}ss, and K.~Schade.
\newblock Dynamics of nematic liquid crystal flows: {T}he quasilinear approach.
\newblock {\em Ann. Inst. H. Poincar\'e Anal. Non Lin\'eaire}, 33(2):397--408,
  2016.

\bibitem{localin3d}
M.-C. Hong, J.~Li, and Z.~Xin.
\newblock Blow-up criteria of strong solutions to the {E}ricksen--{L}eslie
  system in {$\mathbb{R}^3$}.
\newblock {\em Comm. Partial Differential Equations}, 39(7):1284--1328, 2014.

\bibitem{hong}
M.-C. Hong and Z.~Xin.
\newblock Global existence of solutions of the liquid crystal flow for the
  {O}seen--{F}rank model in {$\mathbb{R}^2$}.
\newblock {\em Adv. Math.}, 231(3-4):1364--1400, 2012.

\bibitem{quasiconvex}
D.~Kinderlehrer and P.~Pedregal.
\newblock Characterizations of {Y}oung measures generated by gradients.
\newblock {\em Arch. Rational Mech. Anal.}, 115(4):329--365, 1991.

\bibitem{KinderlehrerPedregal}
D.~Kinderlehrer and P.~Pedregal.
\newblock Gradient {Y}oung measures generated by sequences in {S}obolev spaces.
\newblock {\em J. Geom. Anal.}, 4(1):59--90, 1994.

\bibitem{rindler}
J.~Kristensen and F.~Rindler.
\newblock Characterization of generalized gradient {Y}oung measures generated
  by sequences in $\f {W}^{1,1}$ and \f{B}\f{V}.
\newblock {\em Arch. Ration. Mech. Anal.}, 197(2):539--598, 2010.

\bibitem{approx}
R.~Lasarzik.
\newblock Approximation and optimal control of dissipative solutions to the
  {E}ricksen�--{L}eslie system.
\newblock {\em Numerical Functional Analysis and Optimization},
  40(15):1721--1767, 2019.

\bibitem{masswertig}
R.~Lasarzik.
\newblock Measure-valued solutions to the {E}ricksen--{L}eslie model equipped
  with the {O}seen--{F}rank energy.
\newblock {\em Nonlin. Anal.}, 179:146--183, 2019.

\bibitem{weakstrong}
R.~Lasarzik.
\newblock Weak-strong uniqueness for measure-valued solutions to the
  {E}ricksen--{L}eslie model equipped with the oseen--frank free energy.
\newblock {\em J. Math. Anal. Appl.}, 470(1):36 -- 90, 2019.

\bibitem{leslie}
F.~M. Leslie.
\newblock Some constitutive equations for liquid crystals.
\newblock {\em Arch. Ration. Mech. Anal.}, 28(4):265--283, 1968.

\bibitem{linliu1}
F.-H. Lin and C.~Liu.
\newblock Nonparabolic dissipative systems modeling the flow of liquid
  crystals.
\newblock {\em Comm. Pure Appl. Math.}, 48(5):501--537, 1995.

\bibitem{linliu2}
F.-H. Lin and C.~Liu.
\newblock Partial regularity of the dynamic system modeling the flow of liquid
  crystals.
\newblock {\em Discrete Contin. Dynam. Systems}, 2(1):1--22, 1996.

\bibitem{linliu3}
F.-H. Lin and C.~Liu.
\newblock Existence of solutions for the {E}ricksen--{L}eslie system.
\newblock {\em Arch. Ration. Mech. Anal.}, 154(2):135--156, 2000.

\bibitem{magnes}
J.-L. Lions and E.~Magenes.
\newblock {\em Probl\`{e}mes aux limites non homog\`{e}nes et applications.
  Volume 1}.
\newblock Dunod, Paris, 1968.

\bibitem{LionsBoltzman}
P.-L. Lions.
\newblock Compactness in {B}oltzmann's equation via {F}ourier integral
  operators and applications. {I}, {II}.
\newblock {\em J. Math. Kyoto Univ.}, 34(2):391--427, 429--461, 1994.

\bibitem{lionsfluid}
P.-L. Lions.
\newblock {\em Mathematical topics in fluid mechanics. {V}ol. 1}.
\newblock The Clarendon Press, New York, 1996.

\bibitem{mclean}
W.~McLean.
\newblock {\em Strongly elliptic systems and boundary integral equations}.
\newblock Cambridge University Press, Cambridge, 2000.

\bibitem{mielke}
A.~Mielke.
\newblock On evolutionary {$\varGamma$}-convergence for gradient systems.
\newblock In A.~Muntean, J.~Rademacher, and A.~Zagaris, editors, {\em
  Macroscopic and large scale phenomena: coarse graining, mean field limits and
  ergodicity}, volume~3 of {\em Lect. Notes Appl. Math. Mech.}, pages 187--249.
  Springer, 2016.

\bibitem{olmsted}
P.~D. Olmsted and P.~Goldbart.
\newblock Theory of the nonequilibrium phase transition for nematic liquid
  crystals under shear flow.
\newblock {\em Phys. Rev. A}, 41:4578--4581, 1990.

\bibitem{RoubicekMeasure}
T.~Roub{\'{\i}}{\v{c}}ek.
\newblock {\em Relaxation in optimization theory and variational calculus},
  volume~4 of {\em de Gruyter Series in Nonlinear Analysis and Applications}.
\newblock Walter de Gruyter \& Co., Berlin, 1997.

\bibitem{Sensitivity}
T.~Roub{\'{\i}}{\v{c}}ek.
\newblock Optimality conditions for nonconvex variational problems relaxed in
  terms of {Y}oung measures.
\newblock {\em Kybernetika}, 34(3):335--347, 1998.

\bibitem{roubicek}
T.~Roub{\'{\i}}{\v{c}}ek.
\newblock {\em Nonlinear partial differential equations with applications}.
\newblock Birkh{\"a}user, Basel, 2005.

\bibitem{hoff}
T.~Roub{\'{\i}}{\v{c}}ek and K.-H. Hoffmann.
\newblock About the concept of measure-valued solutions to distributed
  parameter systems.
\newblock {\em Math. Meth. Appl. Sci.}, 18(9):671--685, 1995.

\bibitem{opttrans}
C.~Villani.
\newblock {\em Optimal transport}.
\newblock Springer, Berlin, 2009.

\bibitem{viscoelsticdiff}
D.~A. Vorotnikov.
\newblock Dissipative solutions for equations of viscoelastic diffusion in
  polymers.
\newblock {\em J. Math. Anal. Appl.}, 339(2):876 -- 888, 2008.

\bibitem{recent}
W.~Wang, P.~Zhang, and Z.~Zhang.
\newblock Well-posedness of the {E}ricksen--{L}eslie system.
\newblock {\em Arch. Ration. Mech. Anal.}, 210(3):837--855, 2013.

\bibitem{zeidler3}
E.~Zeidler.
\newblock {\em Nonlinear Functional Analysis and its Applications: III:
  Variational Methods and Optimization}.
\newblock Springer, New York, 2013.

\end{thebibliography}

\end{document}